\documentclass[conference]{IEEEtran}
\IEEEoverridecommandlockouts
\usepackage{cite}
\usepackage{amsmath,amssymb,amsfonts}
\usepackage{algorithmic}
\usepackage{amsmath}
\usepackage{amsmath,amssymb,bm,indentfirst}
\usepackage[ruled,vlined,linesnumbered]{algorithm2e}
\usepackage{mathrsfs}
\usepackage{amsfonts,amsthm,amssymb,mathrsfs,bbding}
\usepackage{amsthm}
\usepackage{graphicx}
\usepackage{multirow}
\usepackage{fancybox}
\usepackage{url}
\usepackage{color}
\usepackage{setspace}
\usepackage{multirow}
\usepackage{fancybox}
\usepackage{url}
\usepackage{color}
\usepackage{graphicx}
\usepackage{textcomp}
\usepackage{xcolor}
\begin{document}

\title{A novel view: edge isoperimetric methods and reliability evaluation of several kinds of conditional edge-connectivity of interconnection networks\thanks{This work was supported by the Doctoral Startup Foundation of Xinjiang University (No. 62031224736), the Science and Technology Project of Xinjiang Uygur Autonomous Region (2020D01C069), Open Project of Applied Mathematics Key Laboratory of Xinjiang (No. 2020D04046), the National Natural Science Foundation of China (No. 11671186), and Tianchi Ph.D Program  (No. tcbs201905)}}

\newtheorem{lem}{Lemma}[section]
\newtheorem{thm}[lem]{Theorem}
\newtheorem{cor}[lem]{Corollary}
\newtheorem{obs}[lem]{Observation}
\newtheorem{rem}[lem]{Remark}
\newtheorem{defi}[lem]{Definition}
\newtheorem{con}[lem]{Conjecture}
\renewcommand{\thefootnote}{\fnsymbol{footnote}}
\renewcommand{\baselinestretch}{1.5}

\author{\IEEEauthorblockN{Mingzu Zhang}
\IEEEauthorblockA{\textit{\footnotesize{School of Mathematics and System Sciences}} \\
\textit{ Xinjiang University}\\
Urumqi, China \\
mzuzhang@163.com}
\and
\IEEEauthorblockN{Zhaoxia Tian}
\IEEEauthorblockA{\textit{\footnotesize{School of Mathematics and System Sciences}} \\
\textit{ Xinjiang University}\\
Urumqi, China \\
zhaoxiattt@163.com}
\and
\IEEEauthorblockN{Lianzhu Zhang}
\IEEEauthorblockA{\textit{\footnotesize{Department of School of
Mathematical Sciences}} \\
\textit{Xiamen University}\\
Fujian, China \\
zhanglz@xmu.edu.cn}}

\maketitle

\begin{abstract}
Reliability evaluation and fault tolerance of an interconnection network of some parallel and distributed systems are discussed separately under various link-faulty hypotheses in terms of different $\mathcal{P}$-conditional edge-connectivity. With the help of edge isoperimetric problem's method in combinatorics, this paper mainly offers a novel and unified view to investigate the $\mathcal{P}$-conditional edge-connectivities of hamming graph $K_{L}^{n}$ with satisfying the property that each minimum $\mathcal{P}$-conditional edge-cut separates the $K_{L}^{n}$ just into two components, such as $L^{t}$-extra edge-connectivity, $t$-embedded edge-connectivity, cyclic edge-connectivity, $(L-1)t$-super edge-connectivity, $(L-1)t$-average edge-connectivity and $L^{t}$-th isoperimetric edge-connectivity. They share the same values in form of $(L-1)(n-t)L^{t}$ (except for cyclic edge-connectivity), which equals to the minimum number of links-faulty resulting in an $L$-ary-$n$-dimensional sub-layer from $K_{L}^{n}$. Besides, we also obtain the exact values of $h$-extra edge-connectivity and $h$-th isoperimetric edge-connectivity of hamming graph  $K_{L}^{n}$ for each $h\leq L^{\lfloor {\frac{n}{2}} \rfloor}$.
For the case $L=2$, $K_2^n=Q_n$ is $n$-dimensional hypercube. Our results can be applied to more generalized class of networks, called $n$-dimensional bijective connection networks, which contains hypercubes, twisted cubes, crossed
cubes, M\"obius cubes, locally twisted cubes and so on. Our results improve several previous results on this topic.
\end{abstract}


\section{Introduction}
With the emergence of new topological architectures, parallel and distributed processing research is at the heart of developing new systems in order to be able to take advantage of the underlying interconnection network. A thorough understanding of various aspects of parallel and distributed systems is necessary to be able to achieve the performance of the new parallel computers and supercomputers. Fault tolerance and reliability have to be optimized for some obvious reasons when choosing good models for interconnection networks. It is well known that the topological interconnection network of parallel and distributed system is usually modeled as a graph $G=(V, E)$, where the vertices and edges represent processors and physical links between the processors.

There are many topological parameters of $G$ to evaluate the reliability of this interconnection network. It is known that the problem of finding node or edge disjoint paths is closely related to a well-known Menger's theorem. Menger theorem laid a solid foundation for the theory of fault tolerance and reliability evaluation of interconnection network of parallel processing systems. Let $x$ and $y$ be two any distinct vertices of a graph $G$. The minimum size of an $(x, y)$ edge-cut equals the maximum number of edge disjoint $(x,y)$-paths. Menger's theorem is a characterization of the edge-connectivity in finite graphs in terms of the maximum number of edge-disjoint paths that can be found between any two distinct pair of vertices. As the Menger's theorem is often required to find a maximum of edge-disjoint paths between two given vertices of $G$, motivated by this, we want to go even further, and consider the cases on many-to-many edge disjoint paths of a connected graph $G$ under additional condition.

Let $\mathcal{P}$ be some graph-theoretic property of a connected graph $G$. We aim to find the maximum number of edge disjoint paths connecting any two disjoint connected subgraphs with just satisfying the property $\mathcal{P}$ in $G$.
If $\mathcal{P}$ is containing a vertex, then the corresponding version is the classical Menger edge-connectivity. Conditional connectivity introduced by F. Harary in 1983, generalized the theories of connectivity in both vertex and edge versions. Not only does it meet the increasing need of a more accurate measure of reliability of large-scale parallel processing systems, but also theoretically enriches the theory of network connectedness. This paper mainly studies six kinds of conditional edge-connectivities under various link-faulty hypotheses. The relationship between the vertex version and various diagnosability has been investigated \cite{cdh2015,hh2012,lxc,lxz2016,lxz,lxz2015,lzxw2015,xlzh2016}.

Instead of focusing on studying the $\mathcal{P}$-conditional edge-connectivity of interconnection networks under various kinds of links faulty hypotheses case by case, this paper mainly offers a novel and unified edge isoperimetric methord to investigated the $\mathcal{P}$-conditional edge-connectivity of hamming graph $K_L^n$ under six kinds of link-faulty hypotheses.  These edge-connectivities shares a common  bipartite property: each minimum $\mathcal{P}$-conditional edge-cut separates the $K_L^n$ into two components (or two parts). For $L$-ary-$n$-dimension hamming graph, $L^t$-extra edge-connectivity, $t$-embedded edge-connectivity, cyclic edge-connectivity, $(L-1)t$-super edge-connectivity, $(L-1)t$-average edge-connectivity and $L^{t}$-th isoperimetric edge-connectivity share the same values in form of $(L-1)(n-t)L^{t}$ for $L\geq 2$, $n\geq1$ and $0\leq t\leq n-1$. Our results are applied to general class of networks. Reliability evaluation and fault tolerance of an interconnection network of some parallel and distributed systems are discussed separately under various link-faulty hypotheses in terms of different $\mathcal{P}$-conditional edge-connectivity.

The rest of this paper is organized as follows. In section II, some preliminaries, terminologies, and six kinds of definitions about $\mathcal{P}$-conditional edge-connectivity will be given. In section III, known results about isoperimetric problem, the exact values of $\xi_m(K_L^n)$ and its construction will be presented. In section IV, some useful properties of the functions $\xi_m(K_L^n)$ and $ex_m(K_L^n)$ will be proved. In section V, unified method for $\mathcal{P}$-conditional edge-connectivities of hamming graph $K_L^n$ will be introduced. Finally, the conclusion of this paper will be given in section VI.

\section{Preliminaries}
In a connected graph $G$, for any edge set $F \subseteq E(G)$, the notation $G-F$ denotes the subgraph obtained after removing the edges in $F$ from $G .$ Given a vertex set $X \subseteq V(G)$, we denote $G[X]$ the subgraph of $G$ induced by $X$, and $\overline{X}=V(G) \backslash X$ is the complement of $X .$ For two vertex sets $X$ and $\overline{X}$, we denote $[X, \overline{X}]$ the set edges of $G$ with one end in $X$ and the other end in $\bar{X}.$
A component of
graph $G$ is a maximal subgraph in which each pair of vertices is connected with each other via a path.

\begin{defi}\cite{Harary} Let $\mathcal{P}$ be a property of $G$. The edge subset $F \subset E(G)$ is defined as a $\mathcal{P}$-conditional edge-cut of $G$, if any, $G-F$ is disconnected, if any, and each component satisfies the condition $\mathcal{P}$. $\lambda(\mathcal{P},G), \mathcal{P}$-conditional edge-connectivity of $G$, is defined as the minimum cardinality $\mathcal{P}$-conditional edge-cut $F$ of $G$.
\end{defi}
Based on the different properties and the faulty-free set, recently many researchers have investigated the various kinds of $\mathcal{P}$-conditional edge-connectivity of many classes of networks, such as $h$-extra edge-connectivity, $l$-embedding edge-connectivity, cyclic edge-connectivity, $k$-super edge-connectivity, $k$-average edge-connectivity and $h$-isoperimetric edge-connectivity.
Let
\par~~~~~~~$\mathcal{P}_{1}^{h}=\{$ containing at least $h$ vertices$\}$,
\par~~~~~~~$\mathcal{P}_{2}^{l}=\{$ lying in an $l$-dimensional subnetwork of $G\}$,
\par~~~~~~~$\mathcal{P}_{3}^{c}=\{$ containing at least one cycle of $G\}$,
\par~~~~~~~$\mathcal{P}_{4}^{k}=\{$ having the minimum degree at least $k\}$,
\par~~~~~~~$\mathcal{P}_{5}^{k}=\{$ satisfying the average degree at least $k\}$.

Note that if $\mathcal{P}=\mathcal{P}_{1}^{h}$, it gives the definition of $h$-extra edge-connectivity of $G$ first introduced by F\`{a}brega and Foil in 1996.
\begin{defi}\cite{Foil(1996)} For a connected graph $G$, an edge subset $F \subseteq E(G)$ is called as an $h$-extra edge-cut of $G$, if exists, $G-F$ is disconnected and every component of $G-F$ has not less than $h$ vertices. The $h$-extra edge-connectivity of $G$, written as $\lambda_{h}(G)$, is defined as the minimum cardinality among all the $h$-extra edge-cuts.
\end{defi}
And if $\mathcal{P}=\mathcal{P}_{2}^{l}$, it is defined as the $l$-embedded edge-connectivity of an $n$-dimensional recursive network $G_n$, $\eta_{l}(G_{n})$ for $0\leq l\leq n-1$, first introduced by Yang and Wang in 2012 \cite{em}.

A graph $G$ is cyclic edge connected if $G$ can be separated into at least two components with two of them containing a cycle. The cyclic edge-connectivity of $G$, $c\lambda(G)$, is minimum number of edges whose removal disconnects the graph $G$ satisfying that at least two components with one cycle \cite{cc}.

Similarly, if $\mathcal{P}=\mathcal{P}_{4}^{k}$ and $\mathcal{P}=\mathcal{P}_{5}^{k}$, the $k$-super edge-connectivity and $k$-average edge-connectivity of $G$ can be defined as the minimum cardinality of $k$-super edge-cut and the minimum cardinality of $k$-average edge-cut , such that each component of their removal from $G$ has the minimum degree at least $k$ and the average degree at
least $k$, respectively. The $h$-th isoperimetric edge-connectivity $\gamma_{h}(G)$ of $G$ is defined as $$\gamma_{h}(G)=\min \{|[U,\overline{U}]|: U \subset V(G),|U| \geqslant h,|\overline{U}| \geqslant h\},$$ where $|[U,\overline{U}]|$ is the number of edges with one end in $U$ and the other end in $\overline{U}=V \backslash U $ \cite{t1996,z}. Write $$\beta_{h}(G)=\min \{|[U,\overline{U}]|: U \subset V(G),|U|=m\}.$$ A graph $G$ with $\gamma_{j}(G)=\beta_{j}(G)$, $j=1, \ldots, k$ is said to be $\gamma_{h}$-optimal.
Although the definition of $h$-th
isoperimetric edge-connectivity $\gamma_{h}(G)$ of the graph $G$ deletes the connectedness of components from that of $h$-extra edge-connectivity $\lambda_{h}(G)$,
and one does find the situation where two definitions are not equivalent, for the networks involved hamming graph and its related variants in this paper, they do share the same values.

Although the above six kinds of $\mathcal{P}$-conditional edge-connectivities contain its own properties, they share a fact that removal their each minimum $\mathcal{P}$-conditional edge-cut results in exact two components (or two parts).
The $\mathcal{P}$-conditional edge-connectivity is called bipartite, if
each minimum $\mathcal{P}$-conditional edge-cut results in exact two components (or two parts).

Let $\mathcal{B}=\{\mathcal{P}~|$~$\mathcal{P}$-conditional edge-connectivity is bipartite$\}$.
\begin{thm}\label{bipart}
 $\lambda(\mathcal{P}_1^k,G)$, $\lambda(\mathcal{P}_2^k,G)$, $\lambda(\mathcal{P}_3^k,G)$,
$\lambda(\mathcal{P}_4^k,G)$, $\lambda(\mathcal{P}_5^k,G)$ and $\lambda(\mathcal{P}_6^k,G)$\footnote{Strictly speaking, the removal of a minimum $h$-th isoperimetric edge-cut, results in two parts rather than two components. The connectedness of these two parts is not necessary.} are bipartite.
\end{thm}
\begin{proof}
For $i\neq 3$, recall that $\lambda(\mathcal{P}_i^k,G)$ is the minimum number of an edge set of the graph $G$ whose removal disconnects the graph $G$ with all its components satisfying the condition $\mathcal{P}_i^k$. Suppose $F$ is a $\mathcal{P}_i^k$- conditional edge-cut, it can be deduced that the optimal number is obtained only when there are exactly two components produced in $G-F$. In fact, if there exists a minimum $\mathcal{P}_i^k$- conditional edge-cut $F_{0}$, whose removal disconnects the connected graph $G$ with all its components $H_{1}, H_{2}, \ldots, H_{z}, z>2$ satisfying the condition $\mathcal{P}_i^k$, then there exists an integer $b$, $1\leq b\leq z$, and $[V(H_{b}), \overline{V(H_{b})}]$ is also an $\mathcal{P}_i^k$-conditional edge-cut of $G$, where the induced subgraphs by $V(H_b)$ and $\overline{V(H_{b})}$ are connected. Because of $F_0\cap E(H_b)\neq \emptyset$ or $F_0\cap E(G[\overline{V(H_b)}])\neq \emptyset$, $|[V(H_{b}), \overline{V(H_{b})}]|\textless|F_{0}|$, a contradiction. So for each $i\neq 3$, $\lambda(\mathcal{P}_i^k,G)$ is bipartite. If $i=3$, $k=c$, for convince, $\lambda(\mathcal{P}_3^c,G)$ means that each component contains one cycle. The similar proof can be done.
\end{proof}

Given a connected graph, once above six conditional edge-connectivities are well-defined, then set $\{$ $\mathcal{P}_1^h$, $\mathcal{P}_2^l$, $\mathcal{P}_3^c$,
$\mathcal{P}_4^k$, $\mathcal{P}_5^k$, $\mathcal{P}_6^h \}\subseteq \mathcal{B}$.
Their bipartite property is its essential feature for them.

We borrow the conception of Mader's atom \cite{mm1971} to study of $\mathcal{P}_i^k$-conditional edge-connectivities. A vertex set $S$ is called a $\lambda(\mathcal{P}_i^k)$-fragment if $[S, \bar{S}]$ is a minimum $\mathcal{P}_i^k$-conditional edge-cut. Obviously, if $S$ is a $\lambda(\mathcal{P}_i^k)$-fragment, so is $\bar{S}$. A minimum $\lambda(\mathcal{P}_i^k)$-fragment is a $\lambda(\mathcal{P}_i^k)$-atom. Let $A$ be called a $\lambda(\mathcal{P}_i^k)$-atom, then $G[A]$ is connected. $\gamma_{k}$-fragment and $\lambda_{k}$-atom are defined similarly without connected condition constrain.
Write
\par~$\alpha(\mathcal{P}_i^k,G)$ the cardinality of a $\lambda(\mathcal{P}_i^k)$-atom and
\par~$\theta_{\mathcal{P}_i^k}(G)=\min\{|X|$\,$|$\,$G[X]$ satisfying the property $\mathcal{P}_i^k\}.$

It is easy to see that a graph $G$ is $\lambda(\mathcal{P}_i^k)$-optimal if and only if $\alpha(\mathcal{P}_i^k,G)=\theta_{\mathcal{P}_i^k}(G)$. The parameters $\alpha(\mathcal{P}_i^k,G)$ and $\theta_{\mathcal{P}_i^k}(G)$ play an important role in obtaining the exact value of $\lambda(\mathcal{P}_i^k,G)$ and determining its optimality.

This idea drives us to study the minimum possible ``boundary-size" of a set of a given ``size". The classical isoperimetric problem is usually expressed in the form of an inequality ${L_M}^2\geq4\pi A_M$ that relates the length $L_M$ of a closed curve $M$ and the area $A_M$ of the planar region that it encloses in the Euclidean plane $R^2$.
Which closed curve minimizes the length $L$ of the curve with a fixed size of the area $A$? The equality holds if and only if $M$ is a circle.
It was ``known" to the ancient Greeks, but it was not until the 19th century that this was proved vigorously by Karl Weierstrass in a series of lectures in 1870, in Berlin.
Isoperimetric problems are classical objects of the study in mathematics. In general,
they ask for the minimum possible ``boundary-size" of a set of a given ``size", where the exact meaning of these words varies according to the problem.
Making efforts to obtain the exact value of the $\mathcal{P}_i^k$-conditional edge-connectivity drives us to study the minimum possible ``edge boundary-size" of a set of a given ``size" in graph $G$.

For discrete case, a graph $G=(V, E)$, the given ``size" is the ``number of vertices", while the ``boundary-size" of the graph is ``the size of the edge boundary or the vertex boundary". Both edge and vertex versions are related to the $\mathcal{P}$-conditional edge-connectivity and $\mathcal{P}$-conditional connectivity, respectively. This paper mainly focuses on edge version. Let $G=(V,E)$ be a connected graph and $A\subseteq V(G)$. Recall that the definition of $\beta_m(G)=\min\limits_{A\subseteq V,|A|=m} |[A,\overline{A}]|$ in the $h$-th isoperimetric edge-connectivity.
For a given graph $G$, the edge isoperimetric problem of $G$ is to find a subset $A^*\subseteq V(G)$ with $|A^*|=m$ such that $\beta_m(G)=|[A^*,\overline{A^*}]|$ for any $1\leq m\leq |V(G)|$ (introduced by Harper in 1964)\cite{harper1964optimal}.
Even if in Harper's seminal paper of edge isoperimetric problem on hypercube $Q_n=K_2^n$,
the original definition $G[A^*]$ is not required to be connected, his example does satisfy that $K_2^n[A^*]$ is connected. Let $\xi^e_{m}(G)=\min\{|[X, \overline{X}]|: X\subset V(G)$ with $|X|=m \leq \lfloor |V(G)|/2\rfloor$, and $G[X]$ is connected $\}$.
$\xi^e_{m}(G)$ is unilateral connected for each $1\leq m\leq \lfloor|V(G)|/2\rfloor$.

In this paper we want to go further.
Let $\xi_{m}(G)=\min\{|[X, \overline{X}]|:X\subset V(G)$ with $|X|=m \leq\lfloor |V(G)| /2\rfloor$, and both $G[X]$ and $G[\overline{X}]$ are connected$\}$.
$\xi_{m}(G)$ is bilateral connected for each $1\leq m\leq \lfloor|V(G)|/2\rfloor$.
More coincidentally, Harper's seminal example meets the equations $\beta_{m}(K_2^n)=\xi^e_{m}(K_2^n)=\xi_{m}(K_2^n)$ for each $m\leq 2^{n-1}$.
However, it does not always hold for general $G$. Even if we add the unilateral and bilateral connected conditions, the result $\xi_{m}(G)=\xi^e_{m}(G)$ does not always hold for each $1\leq m\leq \lfloor|V(G)|/2\rfloor$.

\begin{figure}[htbp]
\scalebox{0.3}{\includegraphics{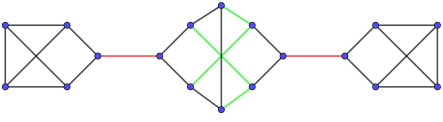}}
\caption{Graph $G^*$ without satisfying condition $\xi_8(G^*)\neq\xi^e_8(G^*)$.~~~~~~}
\label{fig}
\center
\end{figure}
For example, given a 3-regular graph $G^*$ as shown in Fig.1, $\xi_8(G^*)=\lambda_1^8(G^*)=4\neq2$. If the red edge cut in graph $G^*$ is removed, $G^*$ is divided into three components instead of two, and the number of  vertices of the two components is less than 8. If the green edge cut in graph $G^*$ is removed, $G^*$ is divided into two components, the number of vertices is 8 and 10, respectively. So $\xi_8(G^*)=4\neq2=\xi^e_8(G^*)$.

Followed Harper's idea, for $r$-regular graph $G$, $|A|=m$ by Handshaking Lemma, $\beta_m(G)=r|A|-ex_m(G)$, where $ex_m(G)$ is the densest degree sum among all the $m$ vertices induced subgraphs.
For each $m\leq\lfloor|V(G)|/2\rfloor$, if one can find a subset $X_m^*\subseteq V(G)$ satisfying that $\beta_m(G)=\xi^e_m(G)=\xi_m(G)=|[X_m^*,\overline {X_m^*}]|$, $|X_m^*|=m$, and both $G[X^*]$ and $G[\overline {X^*}]$ are connected, with $2|E(G[X_m^*])|=ex_m(G)$, then these three definitions are equivalent.
Such subsets $X_m^*$ are called optimal for edge isoperimetric problem of $G$. We say that optimal subsets are nested $X_m^*\subseteq X_{m+1}^*$ if there exists a total order $\mathcal{O}$ on the set $V(G)$ such that for any $m \leq|V(G)|$ the collection of the first $m$ vertices on the order is an optimal subset
$X_m^*=\{v_{1}, v_{2}, v_{3}, \cdots, v_{i}\}$, $$
v_{1}\textless v_{2}\textless v_{3}\textless\cdots\textless v_{i}|
\textless v_{i+1}\textless\cdots\textless v_{p-2}\textless v_{p-1}\textless v_{p}.$$
For the networks involved hamming graph $K_L^n$ and its related variants in this paper, they do satisfy this property. So in the following discussion, conventionally, we use the notation $\xi_m(K_L^n)$. It is quite convenient for us to study bipartite $\mathcal{P}$-conditional edge-connectivity of $K_L^n$.
In particular, for the $h$-extra edge-connectivity,
in 2018, Zhang et al. obtained that \cite{Zhang(2018)},
for each $1\leq h\leq\lfloor|V(G)|/2\rfloor$, $$\lambda_h(G)=min\{\xi_m(G)~:~h\leq m\leq \lfloor|V(G)|/2\rfloor\}.$$
where
$\xi_{m}(G)=\min\{|[X, \overline{X}]|: X\subset V(G), |X|=m \leq\lfloor |V(G)| /2\rfloor$, and both $G[X]$ and $G[\overline{X}]$ is connected\}.
While for general cases, in 2017, L. P. Montejano and I. Sau proved that the problem of determining the exact value of $\lambda_h(G)$ is NP-complete \cite{ms2017}.

Given two graphs $G$ and $H$, the Cartesian product of $G$ and $H$ $H\times G$ is defined as $V(H\times G)=\{uv\,|\,v\in V(G), u\in V(H)\}$, $E(H\times G)=\{{ux,vy}\,|\,(u=v, {x,y}\in E(G))\;or\;(x=y, {u,v}\in E(H))\}$.

The Cartesian product of graphs $G_3$ and $G_2\times G_1$ is defined as $G_3\times G_2\times G_1$. Similarly, one can define the Cartesian product of graphs $G_n$ and $G_{n-1}\times\cdots\times G_2\times G_1$ inductively.

Each vertex $G_n\times G_{n-1}\times\cdots\times G_2\times G_1$ can be denoted by $x_nx_{n-1}\cdots x_2x_1$, $0\leq x_i\textless|V(G_i)|$, $1\leq i\leq n$.
 If $G_k=G$, for each $1\leq k\leq n$, then $G_n\times G_{n-1}\times\cdots\times G_2\times G_1=\underbrace{G\times G\times\cdots\times G\times G}_n$, denoted by $G^n$, and it is $n$-th cartesian product power graph of $G$. If $G=K_L$, the $n$-th cartesian product power graph of $K_L$ is called $L$-ary $n$-dimensional hamming graph, denoted by $K_L^n$.The vertex set of $K_L^n$ can be denoted by $X_nX_{n-1}\cdots X_2X_1=\{x_nx_{n-1}\cdots x_2x_1\,|\,x_i\in\{0,1\}, 0\leq i\leq 1\}$, $0\leq x_i\leq L-1$, $1\leq i\leq n$. For any two vertices $u=u_nu_{n-1}\cdots u_2u_1$ and $v=v_nv_{n-1}\cdots v_2v_1$ in the graph $K_L^n$,
there is an edge between $u$ and $v$ if and only if exactly one integer $j$ satisfies $u_i\neq v_i$ if $i=j$; $u_i=v_i$ if $i=j$, $j\in\{1, 2, \cdots, n\}$.
By the definition of the $K_L^n$, for given integers $0\leq k\leq L-1$, the subgraph $K_L^n[kX_{n-1}\cdots X_{2} X_{1}]$ induced by $kX_{n-1}\cdots X_{2} X_{1}$ is an $L$-ary $(n-1)$-dimensional sub-layer of $K_L^n$.
Let $z_{n}z_{n-1}\cdots z_{a+2}z_{a+1}X_{a}\cdots X_{2}X_{1}$ be the set \{$z_{n} z_{n-1}\cdots z_{a+2}z_{a+1}x_{a}\cdots x_{2}x_{1}:$ $x_{i}\in \{0, 1, 2, \cdots, L-1 \}$, $i\le 1, 2, \cdots, a$, $z_{j}$ is fixed for each $j=a+1,a+2,\cdots,n $\}. By the definition of the $K_L^n$, similarly,
one can write $K_L^n[z_{n}z_{n-1}\cdots z_{a+2}z_{a+1}X_{a}\cdots X_{2}X_{1}]$ the subgraph induced by the vertex set $z_{n}z_{n-1}\cdots z_{a+2}z_{a+1}X_{a}\cdots X_{2}X_{1}$. It is an $L$-ary $a$-dimensional sub-layer of $K_L^n$.
Thus we also use $z_nz_{n-1}\cdots z_{a+1}X_a\cdots X_2X_1$ to denote this $L$-ary $a$-dimensional sub-layer of $K_L^n$ if no confusion arises.
The lexicographic ordering $\preceq$ on $\Pi_{i=1}^n[k_i]$ is defined as follows:  $x_1x_2\dots x_n$
precedes $y_1y_2\dots y_n$
if for some $i$ we have $x_i<y_i$ and $x_j=y_j$ for $j<i$.
\section{The exact values of $\xi_m(K_L^n)$ and its construction}
Recall that a string that contains only $0's$, $1's$, $\cdots$, $(L-1)'s$ is called an $L$-base $n$-string. For a positivity integer $x=\sum_{i=1}^{n}x_iL^{i-1}$, $0\leq i\leq n$, $x_i\in\{0, 1, 2, \cdots , L-1\}$. Let $x=x_nx_{n-1}\cdots x_1$ be the form of $L$-base $n$-string. The $L$-base $n$-string denotes the $L$-base string of length $n$. For any positive integer $m$, it has a $L$-base decomposition $m=\sum_{i=0}^sa_iL^{b_i}\leq\lfloor{L^n/2}\rfloor$, $a_i\in\{1,2, \cdots , L-1\}$, $b_{0}\textgreater b_1\textgreater\cdots\textgreater b_s$. Let $S_m=\{0, 1, \cdots, m-1\}$ (under decimal representation), $|S_m|=m$, $m\leq\lfloor{L^n/2}\rfloor$, and the corresponding set $L_m^n$ be the corresponding set represented by $L$-base $n$-string. In Table \ref{tab1}, given a positive integer $m$ as an example, the calculation process of $\xi_m(G)$ is explained more concretely and vividly. At the same time, the corresponding induced subgraphs and the functions of $O(K_L)=L^{\lfloor{L\over2}\rfloor}$, for graphs $Q_4$, $FQ_4$, $AQ_4$, $K_3^n$ $K_4^n$ and $K_{10}^n$ are given. In the expression of $ex_m(G^n)$, the specific values of $I_i$ and $\delta_i$ are shown in the Table \ref{tab2}.

\begin{table*}[ht]\scriptsize
\caption{The exact values of $\xi_m(K_L^n)$ and $ex_m(K_L^n)$ and their construction }
\begin{tabular}{p{1.2cm}p{1.5cm}p{1.5cm}p{1.5cm}p{2.2cm}p{2.5cm}p{2cm}p{2.5cm}}
\hline \hline
$G^n$ &$Q_4$ & $FQ_4$ & $AQ_4$ &$K_3^4$ & $K_4^4$ &$K_5^2$&$K_{10}^2$\\
$m$&5 &5&5 &8&10&7&12\\
$S_m$& \{0, 1, 2, 3, 4\}  & \{0, 1, 2, 3, 4\} & \{0, 1, 2, 3, 4\}& \{0, 1, 2, 3, 4, 5, 6, 7\}& \{0, 1, 2, 3, 4, 5, 6, 7, 8, 9\}& \{0, 1, 2, 3, 4, 5, 6\}& \{0, 1, 2, 3, 4, 5, 6, 7, 8, 9, 10, 11\}  \\
$L_m^n$   &\{0000, 0001, 0010, 0011, 0100\}& \{0000, 0001, 0010, 0011, 0100\} & \{0000, 0001, 0010, 0011, 0100\} & \{0000, 0001, 0002, 0010, 0011, 0012, 0020, 0021\}& \{0000, 0001, 0002, 0003, 0010, 0011, 0012, 0013, 0020, 0021\}& \{00, 01, 02, 03, 04, 10, 11\}& \{00, 01, 02, 03, 04, 05, 06, 07, 08, 09, 10, 11\} \\
Decomposition of $m$   &$5=2^2+2^0$&$5=2^2+2^0$&$5=2^2+2^0$&$8=2\times3^1+2\times3^0$&$10=2\times4^1+2\times4^0$&$7=5^1+2\times5^0$&$12=10^1+2\times10^0$ \\
$ex_m(G^n)$    &$2\times2^2+2\times2^0=10$&$2\times2^2+2\times2^0=10$&$3\times2^2-2^0+4\times2^0+1=16$&$2(2\times3^1+3^1+3^0)+2\times2\time2\times3^0=28$&$3\times2\times4^1+2\times4^1+2\times2\times4^0+2\times2\times2\times4^0=42$&$4\times5^1+2\times4^1+2\times5^0+2\times2\times5^0=42$&$9\times10^1+2\times4^1+2\times1\times10^0+2\times2\times10^0=96$\\
$\xi_m(G^n)$&$4\times5-10=10$&$5\times5-10=15$&$(2\times4-1)\times5-16=19$&$2\times4\times8-28=36$&$3\times4\times10-42=78$&$4\times2\times7-42=14$&$9\times2\times12-96=120$\\
Induced graph $G^n[L_m^n]$ &\begin{minipage}[b]{0.2\columnwidth}\flushleft\raisebox{-.5\height}{\includegraphics[width=1.6cm,height=2.2cm]{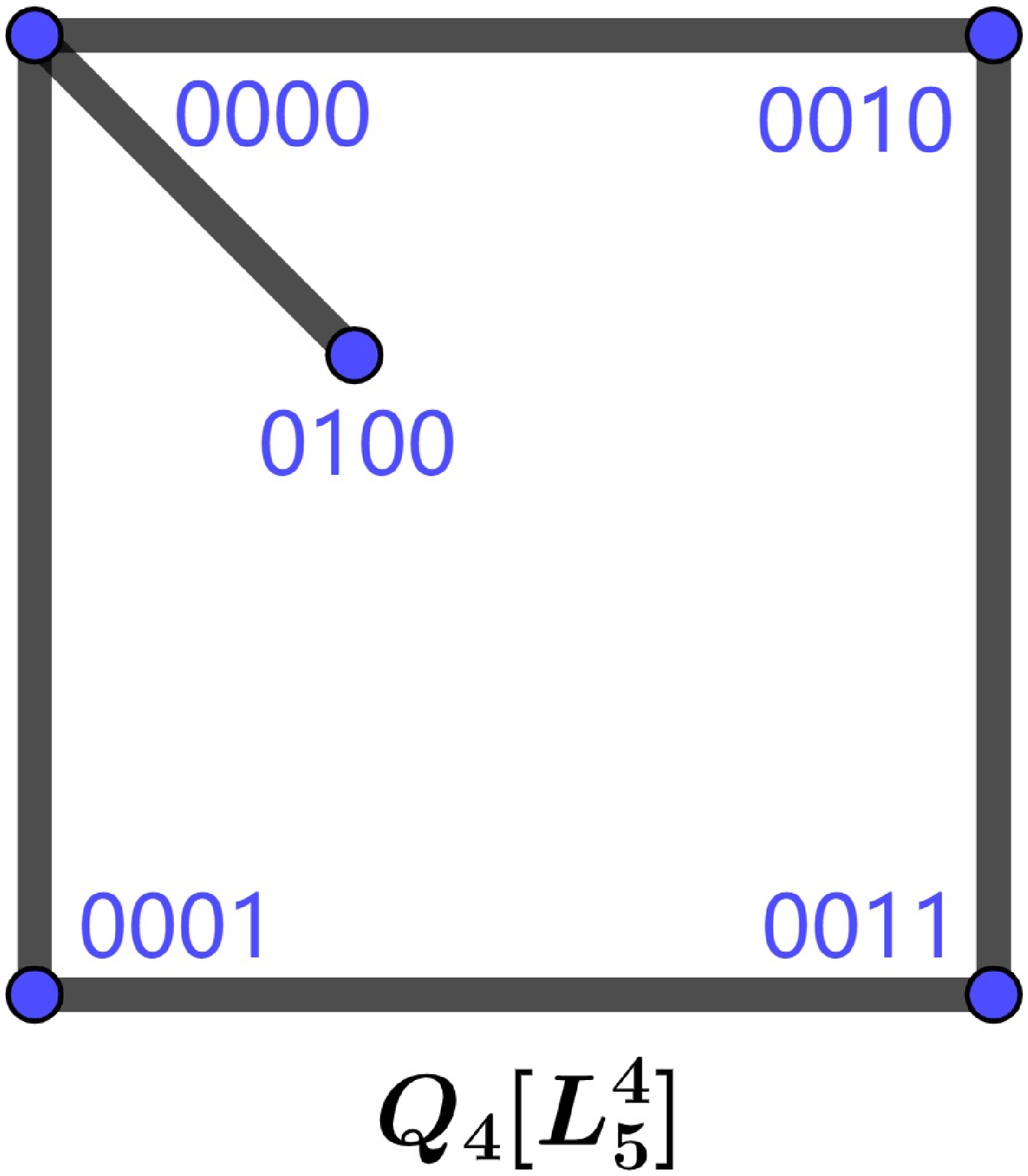}}\end{minipage}&\begin{minipage}[b]{0.2\columnwidth}\flushleft\raisebox{-.5\height}{\includegraphics[width=1.6cm,height=2.2cm]{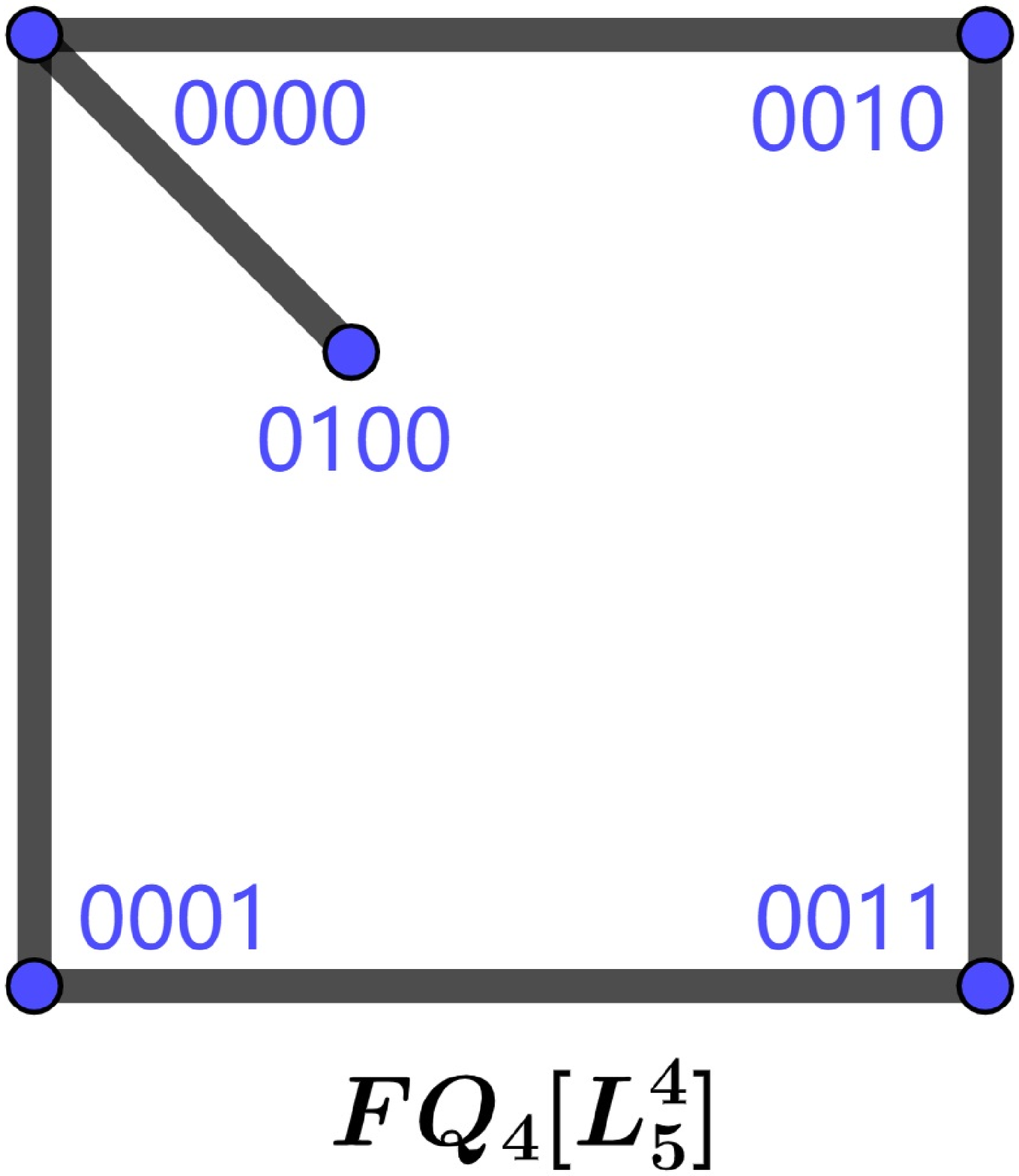}}\end{minipage}&\begin{minipage}[b]{0.2\columnwidth}\flushleft\raisebox{-.5\height}{\includegraphics[width=1.6cm,height=2.5cm]{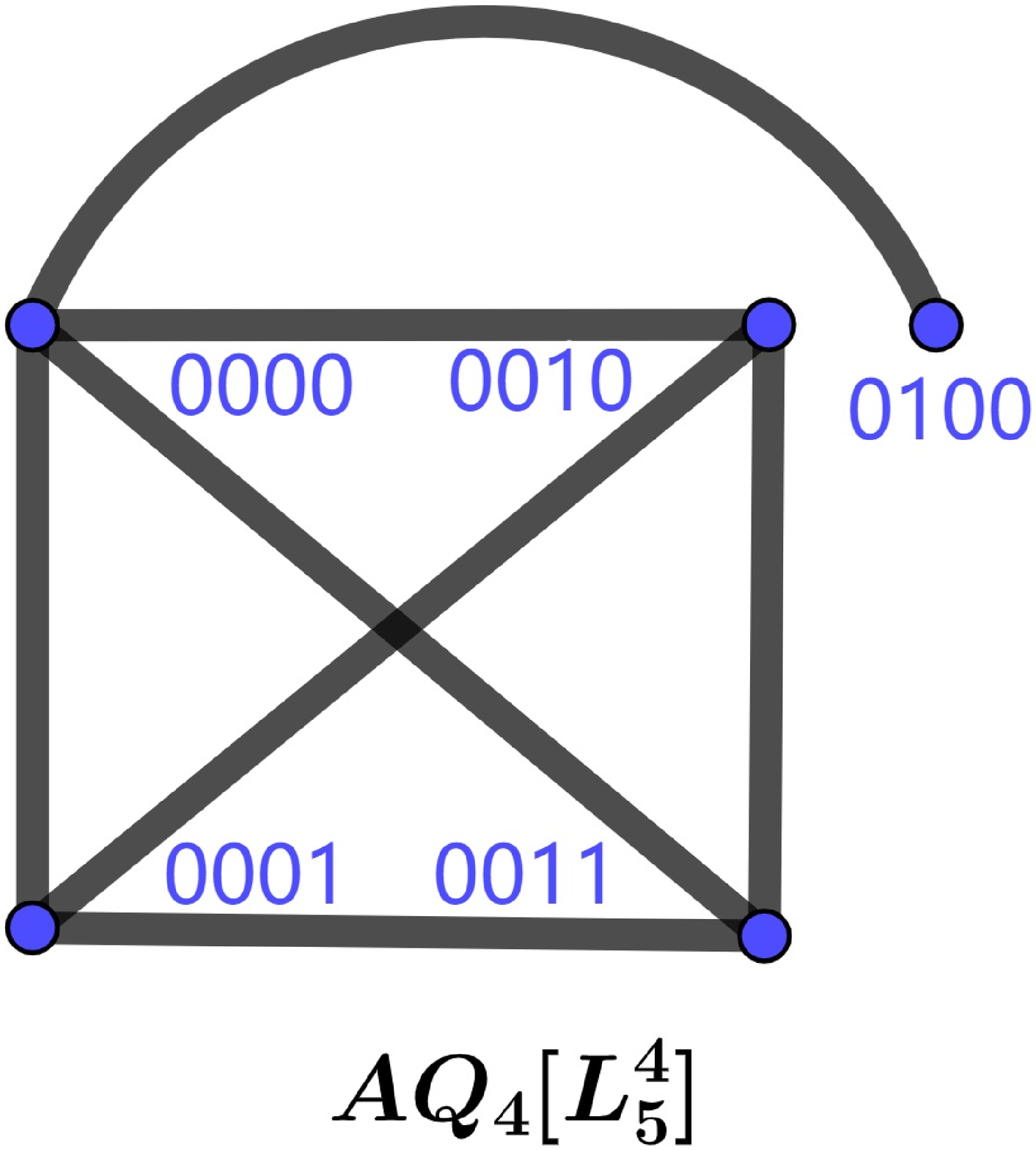}}\end{minipage}&\begin{minipage}[b]{0.2\columnwidth}\flushleft\raisebox{-.5\height}{\includegraphics[width=2.2cm,height=3cm]{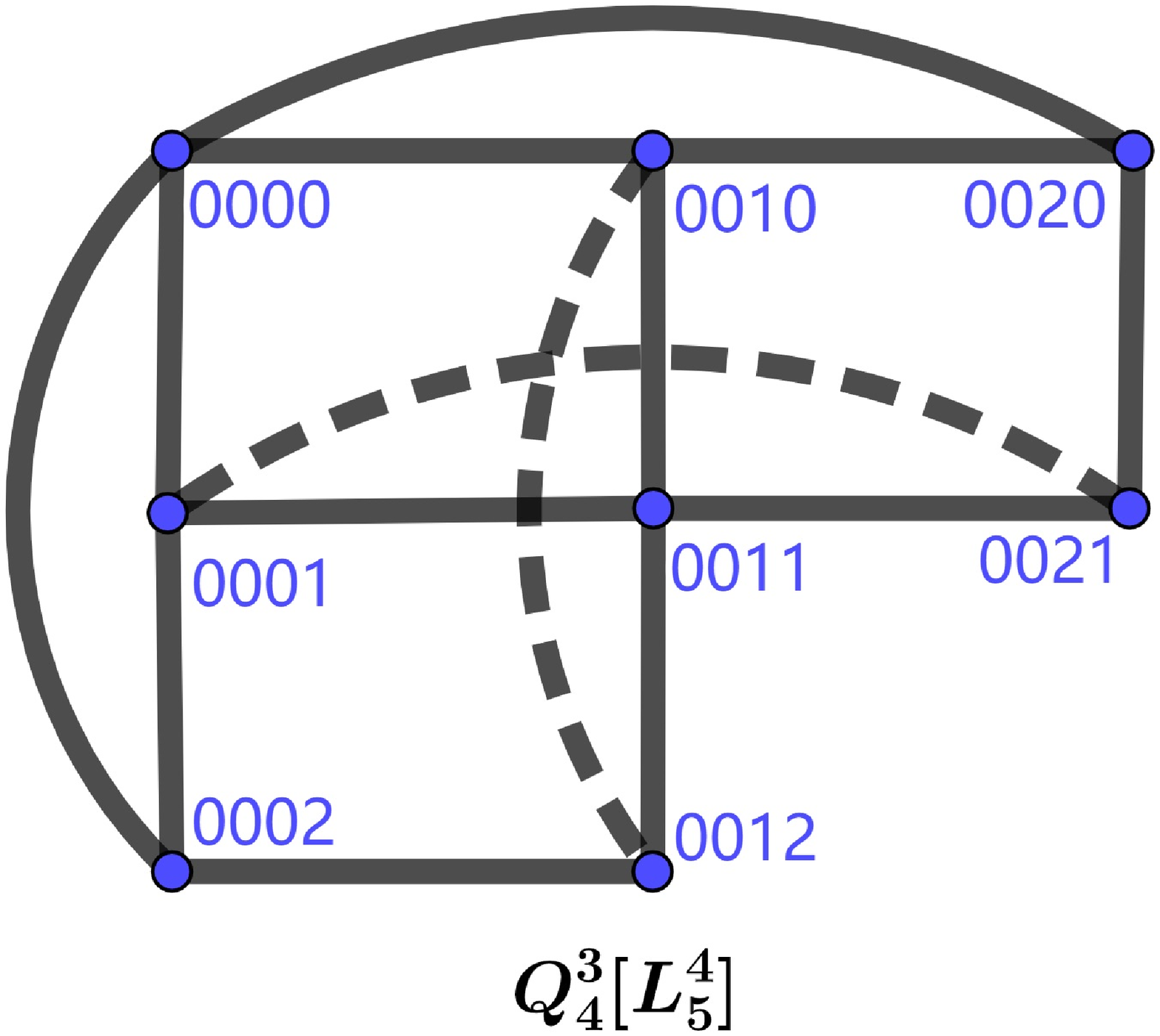}}\end{minipage}&\begin{minipage}[b]{0.2\columnwidth}\flushleft\raisebox{-.5\height}{\includegraphics[width=2.6cm,height=3.2cm]{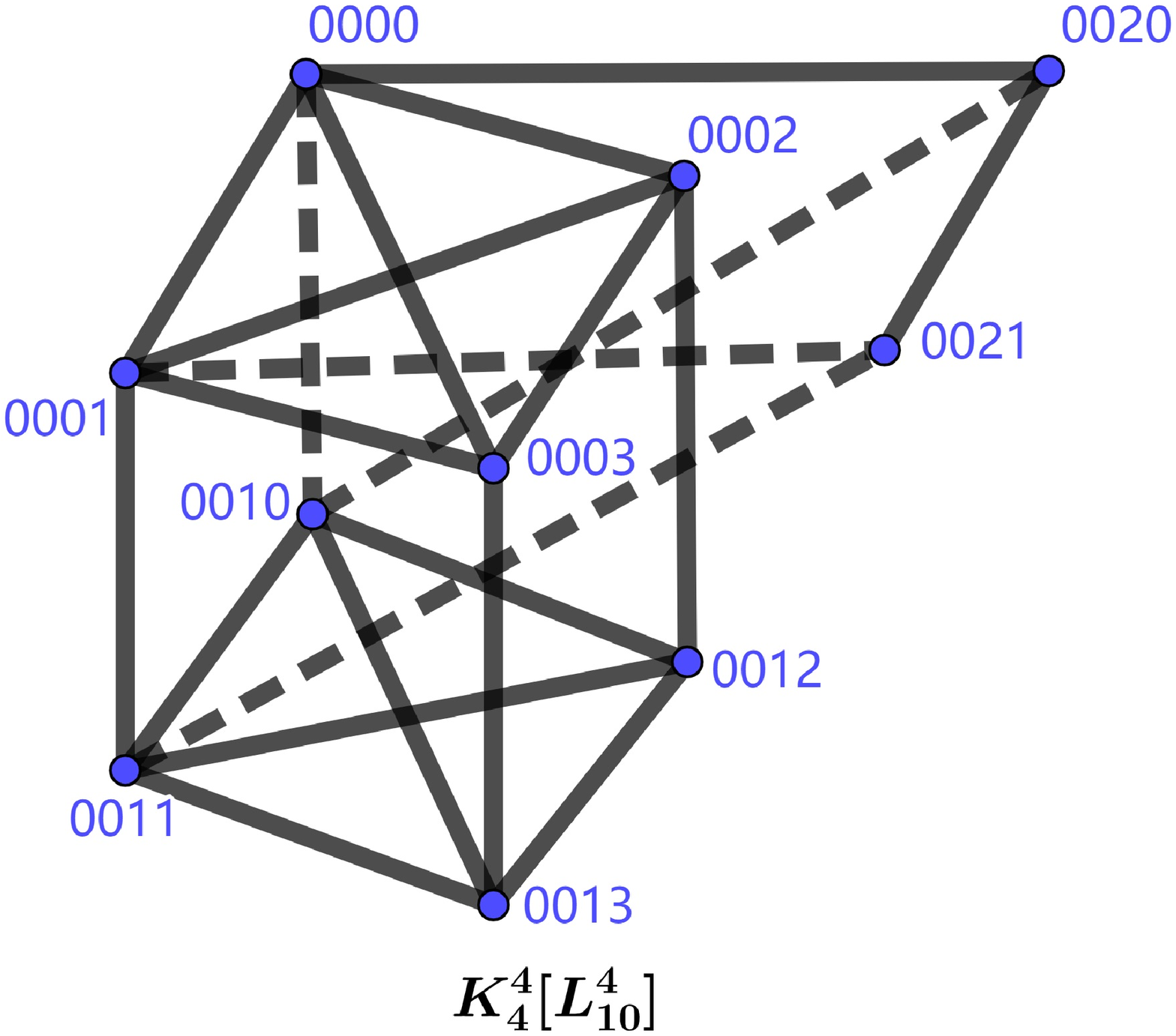}}\end{minipage}&\begin{minipage}[b]{0.2\columnwidth}\flushleft\raisebox{-.5\height}{\includegraphics[width=2cm,height=2.8cm]{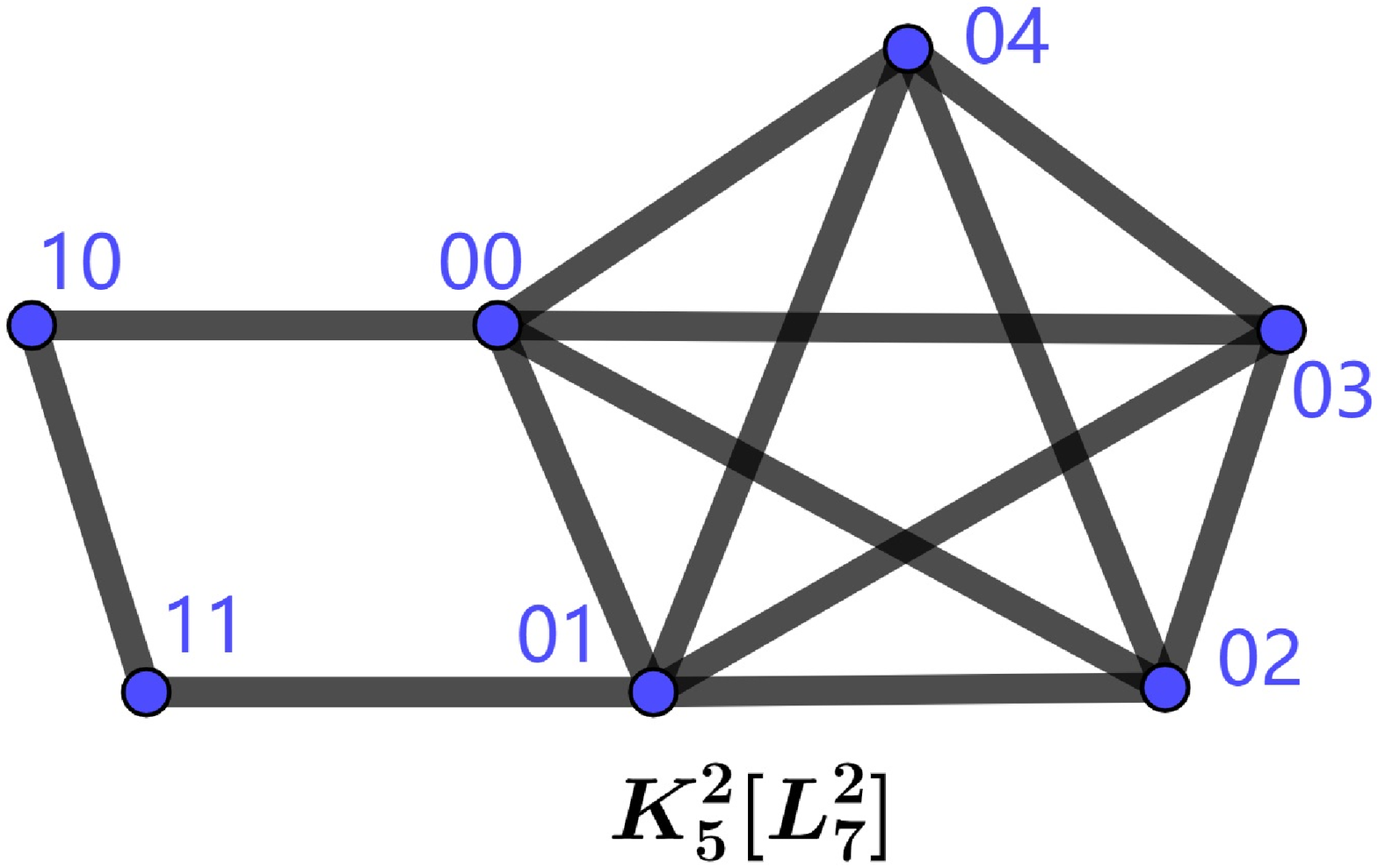}}\end{minipage}&\begin{minipage}[b]{0.2\columnwidth}\flushleft\raisebox{-.5\height}{\includegraphics[width=3cm,height=3.5cm]{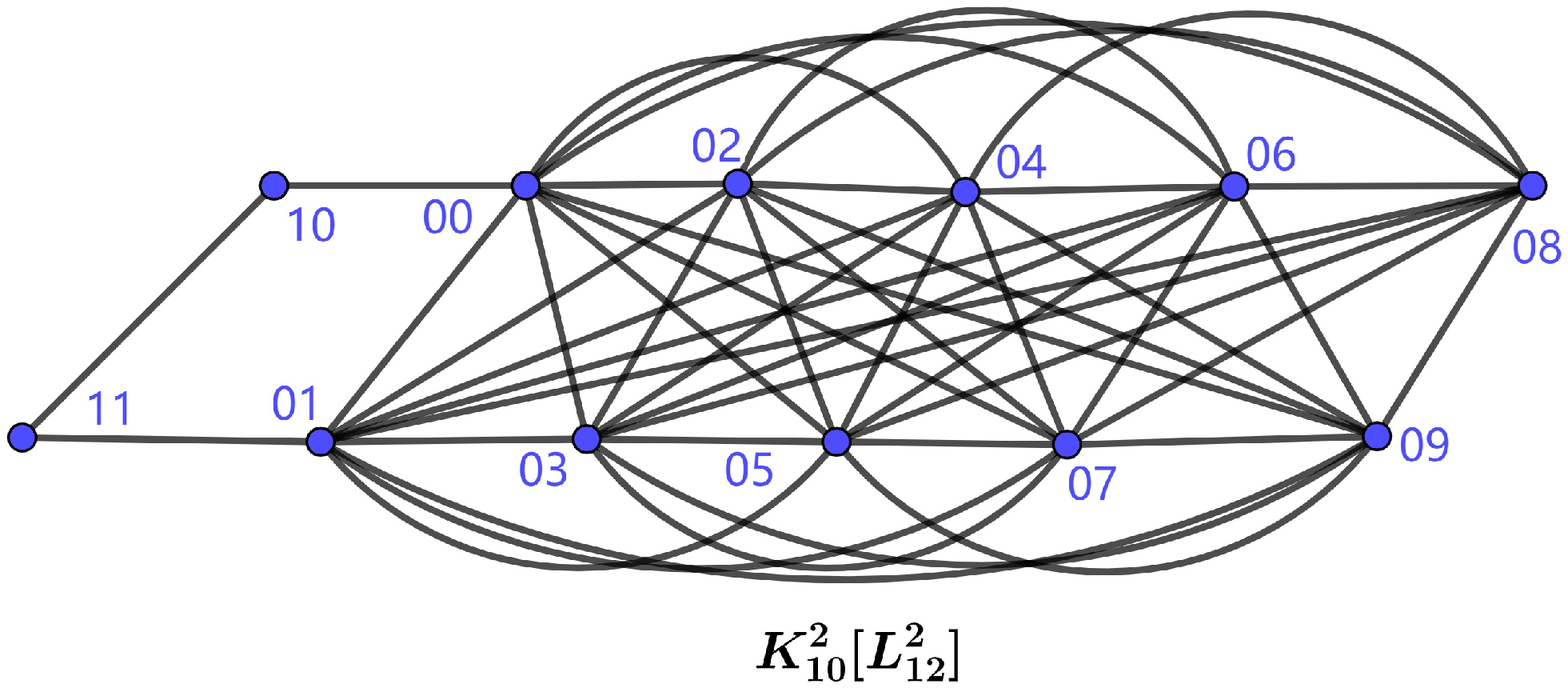}}\end{minipage}\\
$O(K_L)$&$2^{\lfloor\frac{n}{2}\rfloor}$&$2^{\lfloor\frac{n}{2}\rfloor}$&$2^{\lfloor\frac{n}{2}\rfloor}$&$3^{\lfloor\frac{n}{2}\rfloor}$&$4^{\lfloor{n\over2}\rfloor}$&$5^{\lfloor{n\over2}\rfloor}$&$10^{\lfloor{n\over2}\rfloor}$\\
$I(K_L)$&$I_0=0$, $I_1=0$&$I_0=0$, $I_1=0$&$I_0=0$, $I_1=0$&$I_0=0$, $I_1=0$, $I_2=1$&$I_0=0$, $I_1=0$, $I_2=1$, $I_3=3$&$I_0=0$, $I_1=0$, $I_2=1$, $I_3=3$, $I_4=6$&$I_0=0$, $I_1=0$, $\cdots$, $I_9=36$\\
$\delta(K_L)$&$\delta_1=0$, $\delta_2=1$&$\delta_1=0$, $\delta_2=1$&$\delta_1=0$, $\delta_2=1$&$\delta_1=0$, $\delta_2=1$, $\delta_3=2$&$\delta_1=0$, $\delta_2=1$, $\delta_3=2$, $\delta_4=3$&$\delta_1=0$, $\delta_2=1$, $\delta_3=2$, $\delta_4=3$, $\delta_5=4$&$\delta_1=0$, $\delta_2=1$, $\cdots$, $\delta_{11}=10$\\
\hline \hline
\end{tabular}
\footnotesize{$^{\rm *}$Strictly speaking, $ex_m(AQ_n)$ does not conform to this regular pattern, $
ex_{m}(AQ_{n})=\sum_{i=0}^{s}(2 t_{i}-1) 2^{t_{i}}+\sum_{i=0}^{s}4i 2^{t_{i}}+\delta$, where if $m$ is even, then $\delta=0$; if $m$ is odd, then $\delta=1$.}
\label{tab1}
\end{table*}

\begin{table*}[htbp]\scriptsize
\caption{The functions of $I_i$ and $\delta_i$ corresponding to $K_L$.}
\begin{center}
\begin{tabular*}{\hsize}{@{}@{\extracolsep{\fill}}lllllllllll}
\hline
\hline
$I_0=0$&$I_1=0$&$I_2=1$&$I_3=3$&$I_4=6$&$\cdots$&$I_i={i(i-1)/2}$&$\cdots$&$I_{L-1}={(L-2)(L-1)/2}$&$I_{L}={L(L-1)/2}$\\
&$\delta_1=0$&$\delta_2=1$&$\delta_3=2$&$\delta_4=3$&$\cdots$&$\delta_i=i-1$&$\cdots$&$\delta_{L-1}=L-2$&$\delta_L=L-1$\\
\hline
\hline
\end{tabular*}
\label{tab2}
\end{center}
\end{table*}

\begin{lem}\cite{Zhang}
Given $d$-regular connected graph $G$, the number of vertices is $|V(G)|=L$. If the lexicographic order provides the optimal solution of the edge isoperimetric problem on the power graph and $I_0(G)=0$, $I_1(G)=0$, $I_m(G)=ex_m(G)/2$, $\delta(G)=I_m(G)-I_{m-1}(G)$, $1\leq m\leq L$, then for any $1\leq m=\sum_{i=0}^sa_iL^{b_i}\leq L^n$, $b_0\textgreater b_1\textgreater\cdots\textgreater b_s$, $a_i\in\{1, 2, \cdots, L-1\}$, $0\leq i\leq s$, $ex_m(G^n)=\sum\nolimits_{i=0}^s[\delta_L(G)a_ib_iL^{b_i}+2I_{a_i}(G)L^{b_i}]$
$+2\sum\nolimits_{i=0}^{s-1}\sum\nolimits_{k=i+1}^s$ $\delta_{a_i+1}(G)a_kL^{b_k},$
and $\xi_m^e(G^n)=\delta_L(G)nm-ex_m(G^n).$
\end{lem}
Let $m=\sum_{i=0}^sa_iL^{b_i}$, $a_i\in\{1, 2, \cdots, L-1\}$, $n\geq b_0\textgreater b_1\textgreater\cdots\textgreater b_s\geq0$. In 2018, Zhang \cite{Zhang} constructed $(s + 1)$ $L$-ary sub-layer family whose vertex sets do not intersect each other. For $0\leq i\leq s$, the $i$-th $L$-ary sub-layer family of $K_L^n$ is denoted by $C^i$, which contains $a_i$ $L$-ary $b_i$-dimensional sub-layers $C^{i, 1}, C^{i, 2}, \cdots, C^{i, a_0}$ of $K_L^n$. For $0 \leq i \leq s$, $1\leq{ j_{i}}\leq{a_{i}}\leq L-1$, the $j_i$-th $L$-ary $b_i$-dimensional sub-layer of $K_L^n$ is denoted by $C^{i,j_i}$. Given $C^0$, which contains $a_0$ $L$-ary $b_0$-dimensional sub-layers, then the $j_0$-th $L$-ary $b_0$-dimensional sub-layer of $K_L^n$ is $C^{0,j_0}$: $$00...00j_0-1X_{b_0}X_{b_0-1}...X_2X_1.$$
The induced subgraphs of each $L$-ary $b_0$-dimensional sub-layer are isomorphic to the power graph $G^{b_0}$ of a graph $G$. They do not intersect each other and are $\delta_L(G)\times b_0$-regular. There are $I_{a_i}(G)L^{b_0}$ edges in the $0$-th $L$-ary sub-layer family.

The structure process is as follows:
{\tiny\begin{align*}
&C^{0,j_0}, 1\leq j_0\leq a_0: \overbrace{\underbrace{00\cdots00}_{if\,any}j_0-1\underbrace{X_{b_0}X_{b_0-1}\cdots X_2X_1}_{b_0}}^{n}\\
&C^{1,j_1}, 1\leq j_1\leq a_1: \overbrace{\underbrace{00\cdots00}_{if\,any}a_0\underbrace{\underbrace{00\cdots00}_{if\,any}j_1-1\underbrace{X_{b_1}X_{b_1-1}\cdots X_2X_1}_{b_1}}_{b_0}}^{n}\\
&C^{2,j_2}, 1\leq j_2\leq a_2: \overbrace{\underbrace{00\cdots00}_{if\,any}a_0\underbrace{\underbrace{00\cdots00}_{if\,any}a_1\underbrace{\underbrace{00\cdots00}_{if\,any}j_2-1\underbrace{X_{b_2}X_{b_2-1}\cdots X_2X_1}_{b_2}}_{b_1}}_{b_0}}^{n}\\
&\cdots.
\end{align*}}

For $C^{i,j_i}, i>0$, the $(b_{i-1}+1)$-th coordinate $(j_{i-1}-1)$ of $C^{i-1,a_i-1}$ is changed to $a_{i-1}$. Let the $(b_i+1)$-th coordinate of $C^{i,j_i}$ be $(j_i-1)$. And let the coordinate of the $(b_i+2)$-th to the $n$-th be 0 expect for the $(b_{i-1}+1)$-th, if any. Through this construction method, $L_{m}=\bigcup\limits_{0 \leq i \leq s, 1 \leq j_{i} \leq a_{i}} V\left(C^{i, j_{i}}\right)$ is obtained.

Let $V\left(\mathcal{C}^{i}\right)=: \bigcup\limits_{1 \leq j_{i} \leq a_{i}} V\left(C^{i, j_{i}}\right)$ and $G_{m}^{1^{*}}=: G^{n}\left[L_{m}\right]$.

Each $L$-ary sub-layer $C^{i,j_i}$ has $\delta_L(G)b_iL^{b_i}/2$ edges. For $0\leq i\leq s$, $1\leq j_i\leq a_i\leq L-1$, given $i$, if $a_i=1$, $C^i$ contains only one $L$-ary $b_i$-sub-layer; if $a_i=2$, then between $C^{i,a_i}$ and $C^{i,a_i-1}$ has a matching of size $L^{b_i}$; and if $a_i\textgreater2$, $a_i$ and $a_t$ are adjacent in the graph $G$, $a_t\textless a_i$, then between $C^{i,a_i}$ and $C^{i,a_t}$ has a matching of size $L^{b_i}$. There are $I_{a_i}(G)$ groups of such matches exist in $C^i$ $L$-ary $b_i$-dimensional sub-layer of $L$-ary $b_i$-dimensional sub-layer family.

For each $0\leq i\leq k\leq s$, there are $\delta_{a_i+1}(G)a_kL^{b_k}$ edges between sub-layer family $V(C^i)$ and sub-layer family $V(C^k)$. It can be counted that $C_m^{1^*}$ contains $\sum_{i=0}^s[0.5\delta_L(G)b_iL^{b_i}$ $+(I_{a_i}(G))L^{b_i}]$ edges inside $C^i$. The number of edges between $C^i$ is $\sum_{i=0}^{s-1}\sum_{k=i+1}^{s}\delta_{a_i+1}(G)a_kL^{b_k}$. Since $G^n$ is $\delta_L(G)n$-regular, from the Handshaking lemma, $\xi_m^e(G^n)=$ $\delta_L(G)nm-ex_m(G^n)$. The number of edges of $G^n[L_m]$ is $\sum_{i=0}^s[0.5\delta_L(G)a_ib_iL^{b_i}+I_{a_i}(G)L^{b_i}]+\sum_{i=0}^{s-1}\sum_{k=i+1}^s\delta_{a_i+1}(G)a_kL^{b_k}$.

As $G=K_L$ is $L-1$-regular connected graph, the number of vertices is $|V(K_L)|=L$. The lexicographic order provides the optimal solution of the edge isoperimetric problem on $K_L^n$ and $I_0(K_L)=0$, $I_1(K_L)=0$, $2I_m(K_L)=ex_m(G)=m(m-1)$, $\delta_m(K_L)=I_m(K_L)-I_{m-1}(K_L)=m-1$, $1\leq m\leq L$. Note that $\delta_L(K_L)
L-1=r$.
\begin{cor}\cite{Zhang}
For any integers $1\leq m=\sum_{i=0}^sa_iL^{b_i}\leq L^n$, $b_0\textgreater b_1\textgreater\cdots\textgreater b_s$, $a_i\in\{1, 2, \cdots, L-1\}$, $0\leq i\leq s$, $ex_m(K_L^n)=\sum\nolimits_{i=0}^s[(L-1)a_ib_iL^{b_i}+(a_i-1)a_iL^{b_i}]$
$+2\sum\nolimits_{i=0}^{s-1}\sum\nolimits_{k=i+1}^s$ $a_ia_kL^{b_k},$
and $\xi_m^e(K_L^n)=(L-1)nm-ex_m(K_L^n).$
\end{cor}
\begin{cor}\label{bdk2}
\par(1)~(Harper 1964 \cite{harper1964optimal} and Li and Yang \cite{Li(2013)} 2013 ) $\xi_m(K_2^n)=nm-ex_m(K_2^n)=nm-[\sum_{i=0}^{s}b_i2^{b_i}+\sum_{i=0}^{s}2\cdot i\cdot2^{b_i}],$ for $m=\sum_{i=0}^{s}2^{b_i}\leq2^{n}$, $b_0\textgreater b_1\textgreater\dots\textgreater b_s$.
\par(2)~$\xi_m(K_3^n)=2nm-ex_m(K_3^n)=2nm-\{\sum_{i=0}^{s}[2a_ib_i3^{b_i}$ $+2(a_{i}-1)3^{b_{i}}]
+2\sum_{i=0}^{s-1}\sum_{j=i+1}^{s} a_{i}a_{j}3^{b_j}\}$, for $1\leq m=\sum_{i=0}^{s} a_i3^{b_i} \leq 3^n$, $b_0\textgreater b_1\textgreater\dots\textgreater b_s$,
$a_i\in\{1,2\}$, $0\leq i\leq s$.
\end{cor}
The $n$-dimensional bijective connection networks $\mathcal{B}_n$ (also called hypercube-like networks) are a class of cube-based networks including several well known interconnection networks like hypercubes (denoted by $Q_{n}$), twisted cubes \cite{tc2} (denoted by $TQ_{n})$, crossed cubes \cite{hld} (denoted by $CQ_{n}$), spined cubes $SQ_n$, parity cubes $PQ_n$, $Z$-cubes $ZQ_n$ \cite{x}, varietal cubes $VQ_n$, M$\ddot{o}$bius cubes \cite{mc} (denoted by $MQ_{n}$), locally twisted cubes \cite{ltc} (denoted by $LTQ_{n}$), restricted hypercube-like networks
$RHLN_n$, generalized cubes $GQ_n$, generalized twisted cubes \cite{tc} (denoted by $GTQ_{n}$) and Mcubes \cite{m} (denoted by $MCQ_{n}$) as members.

In 2013, Zhang et al. obtained that $\xi_m(\mathcal{B}_n)=\xi_m(K_2^n)$ \cite{Zhang(2014)}.
\begin{cor}\label{bdk2}
$\xi_m(\mathcal{B}_n)=\xi_m(K_2^n)$ for each $m\leq2^{n}$.
\end{cor}
So our main results can be used for the $n$-dimensional bijective connection networks.

\section{Some useful properties of the functions $\xi_m(K_L^n)$ and $ex_m(K_L^n)$}

Before given our main results on $\mathcal{P}_i^k$-conditional edge-connectivity of $K_L^n$, we will give some interesting properties of functions functions $\xi_m(K_L^n)$ and $ex_m(K_L^n)$. If there is no specific explanation in the following, we assume that $n$ is a positive integer $n\geq2$.

\begin{lem}\label{fenjie}
Let $1\leq h=h_1+h_2\leq L^{\lfloor{n\over2}\rfloor}$, where $h_1=\sum_{i=0}^{t}a_iL^{b_i}$, $h_2=\sum_{i=0}^{s}a_i^{'}L^{{b_i}^{'}}$, $a_i, a_i^{'}\in\{1, 2, \cdots, L-1\}$, and $ b_0\textgreater b_1\textgreater ...\textgreater b_t\textgreater b_0^{'}\textgreater ...\textgreater b_s^{'}$. Then $ex_h(K_{L}^{n})=ex_{h_1}(K_{L}^{n})+ex_{h_2}(K_{L}^{n})+2(\sum\nolimits_{i=0}^{t}a_i)h_2$.
\end{lem}
\begin{proof}
According to the expression of $ex_m(K_L^n)$, we have $ex_{h_1}(K_{L}^{n})=\sum\nolimits_{i=0}^t[(L-1)a_ib_iL^{b_i}+2I_{a_i}L^{b_i}]$ $+2\sum\nolimits_{i=0}^{t-1}\sum\nolimits_{k=i+1}^{t}a_ia_kL^{b_k}$ and $ex_{h_2}(K_{L}^{n})=\sum\nolimits_{i=0}^{s}[(L-1)a_i^{'}b_i^{'}L^{b_i^{'}}+2I_{a_i^{'}}L^{b_i^{'}}]+2\sum\nolimits_{i=t+1}^{t+s}\sum\nolimits_{k=i+1}^{t+s+1}a_i^{'}a_k^{'}L^{b_k^{'}}$. Let $a_{t+1}=a_0'$, $a_{t+2}=a_1'$, $\cdots$, $a_{t+s+1}=a_s'$, $b_{t+1}=b_0'$, $b_{t+2}=b_1'$, $\cdots$, $b_{t+s+1}=b_s'$. Then $h_2=\sum\nolimits_{i=t+1}^{t+s+1}(a_i^{'}L^{{b_i}^{'}})$ and
$ex_{h_2}(K_{L}^{n})=\sum\nolimits_{i=t+1}^{t+s+1}[(L-1)a_i^{'}b_i^{'}L^{b_i^{'}}+2I_{a_i^{'}}L^{b_i^{'}}]+2\sum\nolimits_{i=t+1}^{t+s}\sum\nolimits_{k=i+1}^{t+s+1}a_i^{'}a_k^{'}L^{b_k^{'}}$. Note that
\par ~$ex_h(K_{L}^{n})$
\par ~~~$=ex_{h_1+h_2}(K_{L}^{n})$
\par ~~~$=\sum\nolimits_{i=0}^t[(L-1)a_ib_iL^{b_i}+2I_{a_i}L^{b_i})+\sum\nolimits_{i=t+1}^{t+s+1}[(L-1)$
\par ~~~~~$\times a_i^{'}b_i^{'}L^{b_i^{'}}+2I_{a_i^{'}}L^{b_i^{'}}]+2\sum\nolimits_{i=0}^{t-1}\sum\nolimits_{k=i+1}^{t}a_ia_kL^{b_k}$
\par ~~~~~$+2\sum\nolimits_{i=t+1}^{t+s}\sum\nolimits_{k=i+1}^{t+s+1}a_i^{'}a_k^{'}L^{b_k^{'}}+2\sum\nolimits_{i=0}^{t}a_i$
\par ~~~~~$\times\sum\nolimits_{i=t+1}^{t+s+1}a_i^{'}L^{{b_i}^{'}}$
\par ~~~$=\sum\nolimits_{i=0}^t((L-1)a_ib_iL^{b_i}+2I_{a_i}L^{b_i})+2\sum\nolimits_{i=0}^{t-1}\sum\nolimits_{k=i+1}^{t}$
\par ~~~~~$(a_ia_kL^{b_k})+\sum\nolimits_{i=t+1}^{t+s+1}[(L-1)a_i^{'}b_i^{'}L^{b_i^{'}}+2I_{a_i^{'}}L^{b_i^{'}}]$
\par~~~~~$+2\sum\nolimits_{i=t+1}^{t+s}\sum\nolimits_{k=i+1}^{t+s+1}a_i^{'}a_k^{'}L^{b_k^{'}}+2\sum\nolimits_{i=0}^{t}a_i$
\par ~~~~~$\times\sum\nolimits_{i=t+1}^{t+s+1}a_i^{'}L^{{b_i}^{'}}$
\par ~~~$=ex_{h_1}(K_{L}^{n})+ex_{h_2}(K_{L}^{n})+2(\sum\nolimits_{i=0}^{t}a_i)h_2$.

Then, the proof is completed.
\end{proof}

\begin{lem}\label{Znz}
$\xi_{m+1}(K_L^n)-\xi_{m}(K_L^n)\geq 0$ for $1\leq m < L^{\lfloor{n\over2}\rfloor}$.
\end{lem}
\begin{proof}
Let $m=\sum_{i=0}^sa_iL^{b_i} < L^{\lfloor{n\over2}\rfloor}, 1\leq a_i\leq L$. So $m=\sum_{i=0}^sa_i< \lfloor{n\over2}\rfloor$. Then
\par $\xi_{m+1}(K_L^n)-\xi_{m}(K_L^n)$
\par ~~~$=\delta_L(K_L^n)n(m+1)-ex_{m+1}(K_L^n)-[\delta_L(K_L^n)nm$
\par ~~~~~$-ex_{m}(K_L^n)]$
\par ~~~$=(L-1)n(m+1)-ex_{m+1}(K_L^n)-[(L-1)nm$
\par ~~~~~$-ex_{m}(K_L^n)]$
\par ~~~$=(L-1)n-[ex_{m+1}(K_L^n)-ex_{m}(K_L^n)]$
\par ~~~$=(L-1)n-2\sum\nolimits_{i=0}^sa_i$
\par ~~~$\geq(L-1)n-2\times\lfloor{n\over2}\rfloor(L-1)$
\par ~~~$=(L-1)(n-2\times\lfloor{n\over2}\rfloor)\geq0$.

The proof is completed.
\end{proof}

\begin{lem}\label{z1} For any $0\leq g\leq L-2$, $0\leq t\leq n-2$, then $\xi_{(g+1)L^t}(K_L^n)-\xi_{gL^t}(K_L^n)\geq 0$.
\end{lem}
\begin{proof}
$\xi_{(g+1)L^t}(K_L^n)-\xi_{gL^t}(K_L^n)=(L-1)nL^t-(L-1)tL^t-2gL^t=(L-1)(n-t)L^t-2gL^t\textgreater(L-1)(n-t-2)L^t$. The last inequality holds because of $g<L-1$ and $n-t\geq2$. So the result holds.
\end{proof}

\begin{lem}\label{Zng}For any $0\leq g\leq L-1$, $0\leq t\leq n-2$, $h_0\textless L^t$, then $\xi_{gL^t+h_0}(K_L^n)-\xi_{gL^t}(K_L^n)\geq 0$.
\end{lem}
\begin{proof}
$\xi_{gL^t+h_0}(K_L^n)-\xi_{gL^t}(K_L^n)=(L-1)nh_0-2gh_0-ex_{h_0}(K_L^n)\geq(L-1)(n-2)h_0-ex_{h_0}(K_L^{n-2})=\xi_{h_0}(K_L^{n-2})\geq0.$
\end{proof}

\begin{lem}\label{Znt}For any positive integers $0\leq t\leq n-2$, then $\xi_{L^{t+1}}(K_L^n)-\xi_{L^{t}}(K_L^n)\geq 0$.
\end{lem}
\begin{proof}
\par ~$\xi_{L^{t+1}}(K_L^n)-\xi_{L^{t}}(K_L^n)$
\par ~~~$=(L-1)n(L-1)L^t-ex_{L_{t+1}}(K_L^n)+ex_{L^t}(K_L^n)$
\par ~~~$=(L-1)^2nL^t-[(L-1)(t+1)L^{t+1}-(L-1)tL^t]$
\par ~~~$=(L-1)L^t(nL-n-tL-L-t)$
\par ~~~$=(L-1)[(n-t)(L-1)-L]L^t$
\par ~~~$\geq(L-1)[2(L-1)-L]L^t$
\par ~~~$=(L-1)(L-2)L^t$
\par ~~~$\geq 0.$
\end{proof}

\begin{lem}\label{Zn1}
For some integers $L$, $m$, $n$, $g$ and $t$, $n\geq 2$, $L\geq 2$, $1\leq g\leq L-1$ and $gL^t\leq  m\leq {L^{n-1}}$, then $\xi_m(K_L^n)\geq\xi_{gL^t}(K_L^n)$.
\end{lem}
\begin{proof}
\textbf {Case 1:}
If $g=1$, $t=n-1$, then $\xi_{L^{n-1}}(K_L^n)\geq \xi_m(K_L^n)=\xi_{L^{n-1}}(K_L^n)$.

\textbf {Case 2:} For $1\leq g\leq g_0\leq L-1$, $t=n-2$, by \textbf{Lemma} \ref{z1}, we have $\xi_{(g_0+1)L^{n-2}}(K_L^n)\geq\xi_{g_0L^{n-2}}(K_L^n)$ for $g_0\leq L-2$. And by \textbf{Lemma} \ref{Zng}, $\xi_{g_0L^{n-2}+h_0}(K_L^n)\geq\xi_{g_0L^{n-2}}(K_L^n)$ for $g_0\leq L-1$ is be obtained. So we can have $\xi_m(K_L^n)\geq\xi_{gL^t}(K_L^n)$ for any $gL^t\leq m\leq L^{n-1}$.

\textbf {Case 3:} For $0\leq t\textless n-2$, we divide the $gL^t\leq m\leq L^{n-1}$ into two parts $m\textless L^{t+1}$ and $L^{t+1}\leq m\leq L^{n-1}$.

Subcase 3.1: By \textbf{Lemma} \ref{Znt}, we can obtain that $\xi_{L^{t_0+1}}(K_L^n)\geq\xi_{L^{t_0}}(K_L^n)$ for any integer $t+1\leq t_0\leq n-2$. And by \textbf{Lemma} \ref{z1} and \ref{Zng}, we can obtain that $\xi_m(K_L^n)\geq\xi_{L^{t+1}}(K_L^n)$ for any $L^{t+1}\leq m\leq L^{n-1}$.

Subcase 3.2: For any $gL^t\leq m\textless L^{t+1}$, $g\leq g_0\leq L-1$, $\xi_{(g_0+1)L^t}(K_L^n)\geq\xi_{g_0L^t}(K_L^n)$ is be obtained by \textbf{Lemma} \ref{z1} and $\xi_{g_0L^t+h_0}(K_L^n)\geq\xi_{g_0L^t}(K_L^n)$ for $h_0\textless L^t$ is be obtained by \textbf{Lemma} \ref{Zng}. So one can have $\xi_m(K_L^n)\geq\xi_{gL^t}(K_L^n)$ for any $gL^t\leq m\textless L^{t+1}$.

Combined Subcase 3.1 and Subcase 3.2, we can deduce that $\xi_m(K_L^n)\geq\xi_{gL^t}(K_L^n)$ for any $gL^t\leq m\leq L^{n-1}$.

In summary, $\xi_m(K_L^n)\geq\xi_{gL^t}(K_L^n)$ for $n\geq 2$, $L\geq 2$, $1\leq g\leq L-1$, $gL^t\leq  m\leq {L^{n-1}}$. The proof is done.
\end{proof}

\begin{lem}\label{Znd1} For $L^{n-1}\leq m\leq \lfloor L^n/2\rfloor$, we have $\xi_m(K_L^n)\geq\xi_{L^{n-1}}(K_L^n)$.
\end{lem}
\begin{proof}
Since the function $\xi_m(K_L^n)$ is highly relied on the decomposition of the integer $m$ and on the parity of $L$, we give the decomposition of the integer $\lfloor L^n/2\rfloor$ as follows:
\begin{equation}
\lfloor L^n/2\rfloor=\left\{
\begin{aligned}
&{L\over2}L^{n-1} &\text{for even}\;L; \\
&\sum\nolimits_{i=0}^{n-1}\lfloor{L\over2}\rfloor L^{n-1-i} &\text{for odd}\;L.\\
\end{aligned}
\right.
\end{equation}
In order to obtain our result, three cases of decomposition of $m$ are discussed (if $L$ is even, there are only {\bf Case 1} and {\bf Case 2}).

{\bf Case 1:} For any $1\leq g< \lfloor{L\over2}\rfloor$, one can obtain that
\par ~$\xi_{(g+1)L^{n-1}}(K_L^n)-\xi_{gL^{n-1}}(K_L^n)$
\par ~~~$=(L-1)nL^{n-1}-(L-1)(n-1)L^{n-1}-2gL^{n-1}$
\par ~~~$\geq(L-1)L^{n-1}-2(\lfloor{L\over2}\rfloor-1)L^{n-1}$
\par ~~~$=(\lceil{L\over2}\rceil-\lfloor{L\over2}\rfloor+1)L^{n-1}$
\par ~~~$\geq L^{n-1}$
\par ~~~$>0.$

{\bf Case 2:} For any $1\leq g<\lfloor{L\over2}\rfloor$, $h_0\textless L^{n-1}$, we have
\par ~$\xi_{gL^{n-1}+h_0}(K_L^n)-\xi_{gL^{n-1}}(K_L^n)$
\par ~~~$=(L-1)nh_0-2gh_0-ex_{h_0}(K_L^n)$
\par ~~~$\geq(L-1)nh_0-2(\lfloor{L\over2}\rfloor-1)h_0-ex_{h_0}(K_L^n)$
\par ~~~$=(L-1)n-(L-1)h_0+(\lceil{L\over2}\rceil-\lfloor{L\over2}\rfloor+1)h_0$
\par ~~~~~$-ex_{h_0}(K_L^n)$
\par ~~~$=(L-1)(n-1)h_0-ex_{h_0}(K_L^n)+(\lceil{L\over2}\rceil-\lfloor{L\over2}\rfloor+1)h_0$
\par ~~~$=\xi_{h_0}(K_L^{n-1})+(\lceil{L\over2}\rceil-\lfloor{L\over2}\rfloor+1)h_0$
\par ~~~$\geq 0.$

Actually, by the {\bf Case 1} and {\bf Case 2}, one can already prove that for $L^{n-1}\leq m\leq \lfloor{L\over2}\rfloor L^{n-1}$, $\xi_m(K_L^n)\geq\xi_{L^{n-1}}(K_L^n)$.
But if $L$ is odd, for $\lfloor{L\over2}\rfloor L^{n-1}\leq m\leq \sum\nolimits_{i=0}^{n-1}\lfloor{L\over2}\rfloor L^{n-1-i}$, two subcases are still needed to discuss.
\par {\bf Case 3:}

Subcase 3.1: For odd $L$, $h_0<{1\over2}L^{n-2-j}$, then

\par ~$\xi_{\sum\nolimits_{i=0}^j\lfloor{L\over2}\rfloor L^{n-1-j}+h_0}(K_L^n)-\xi_{\sum\nolimits_{i=0}^j\lfloor{L\over2}\rfloor L^{n-1-j}}(K_L^n)$
\par ~~~$=(L-1)nh_0-2\sum\nolimits_{i=0}^j\lfloor{L\over2}\rfloor h_0-ex_{h_0}(K_L^n)$
\par ~~~$=(L-1)nh_0-2(\lfloor{L\over2}\rfloor-1)(j+1)h_0-ex_{h_0}(K_L^n)$
\par ~~~$=(L-1)(n-j-1)h_0-ex_{h_0}(K_L^n)$
\par ~~~$=(L-1)(n-j-1)h_0-ex_{h_0}(K_L^{n-j-1})$
\par ~~~$=\xi_{h_0}(K_L^{n-j-1})$
\par ~~~$\geq 0.$

Subcase 3.2:
\par ~$\xi_{\sum\nolimits_{i=0}^{j+1}\lfloor{L\over2}\rfloor L^{n-1-j}}(K_L^n)-\xi_{\sum\nolimits_{i=0}^j\lfloor{L\over2}\rfloor L^{n-1-j}}(K_L^n)$
\par ~~~$=(L-1)n{L\over2}L^{n-1-j-1}-2\sum\nolimits_{i=0}^j\lfloor{L\over2}\rfloor L^{n-2-j}$
\par ~~~~~$-ex_{{L\over2}L^{n-2-j}}(K_L^n)$
\par ~~~$=(L-1)n\lfloor{L\over2}\rfloor L^{n-2-j}-(L-1)(j+1){L\over2}L^{n-2-j}$
\par ~~~~~$-ex_{{L\over2}L^{n-2-j}}(K_L^n)$
\par ~~~$=(L-1)(n-j-1){L\over2}L^{n-2-j}-ex_{{L\over2} L^{n-2-j}}(K_L^{n-1-j})$
\par ~~~$=\xi_{{L\over2}L^{n-2-j}}(K_L^{n-j-1})$
\par ~~~$\geq 0.$

If $L$ is odd, for $\lfloor{L\over2}\rfloor L^{n-1}\leq m\leq \sum\nolimits_{i=0}^{n-1}\lfloor{L\over2}\rfloor L^{n-1-i}$, based on above two subcases, one can obtain that $\xi_{m}(K_L^{n-1})\geq\xi_{\lfloor{L\over2}\rfloor}L^{n-1}(K_L^{n-1})\geq\xi_{L^{n-1}}(K_L^n)$.
All in all, the proof is done.
\end{proof}

\begin{lem}\label{gld}
For some nonnegative integer $t\leq n-1$, $1\leq g\leq L-1$, $gL^t\leq m\leq \lfloor{L^n/2}\rfloor$, then $\xi_m(K_L^n)\geq\xi_{gL^t}(K_L^n)$.
\end{lem}
\begin{proof}
\textbf{Case 1:} For $t\textless n-1$, by \textbf{Lemma} \ref{Znd1} and \textbf{Lemma} \ref{Zn1}, the inequality $\xi_m(K_L^n)\geq\xi_{gL^t}(K_L^n)$ holds for any $gL^t\leq m\leq \lfloor{L^n/2}\rfloor$, $1\leq g \leq t-1$.

\textbf{Case 2:} For $t=n-1$, we need to prove that $\xi_m(K_L^n)\geq\xi_{gL^{n-1}}(K_L^n)$ for any $gL^{n-1}\leq m\leq \lfloor{L^n/2}\rfloor$, $1\leq g \leq \lfloor{L/2}\rfloor$. Although the integer satisfies $1\leq g\leq L-1$, because of $gL^{n-1}\leq \lfloor{L^n/2}\rfloor$, one can obtain that $g\leq \lfloor{L/2}\rfloor$. By the proof of {\bf Case 1} and {\bf Case 2} in \textbf{Lemma} \ref{Znd1}, $\xi_m(K_L^n)\geq\xi_{gL^{n-1}}(K_L^n)$ for any $gL^{n-1}\leq m\leq \lfloor{L/2}\rfloor L^{n-1}$. If $L$ is even , $\lfloor{L^n/2}\rfloor={L\over2}L^{n-1}$, we can obtain that $\xi_m(K_L^n)\geq gL^{n-1}(K_L^n)$ for any $gL^{n-1}\leq m\leq \lfloor{L/2}\rfloor L^{n-1}=\lfloor{L^n/2}\rfloor$, our result holds. While for odd $L$, because of $\lfloor{L^n/2}\rfloor=\sum_{i=0}^{n-1}\lfloor{L/2}\rfloor L^{n-1-i}$, by the proof of Subcase 3.1 and Subcase 3.2 in \textbf{Lemma} \ref{Znd1}, we have $\xi_m(K_L^n)\geq\xi_{\lfloor{L/2}\rfloor L^{n-1}}(K_L^n)$ for any $\lfloor{L/2}\rfloor L^{n-1}\leq m\leq\lfloor{L^n/2}\rfloor$. So for odd $L$, $\xi_m(K_L^n)\geq\xi_{gL^{n-1}}(K_L^n)$ for any $gL^{n-1}\leq m\leq \lfloor{L^n/2}\rfloor$ and $1\leq g\leq \lfloor{L/2}\rfloor$.

In summary, for any integers $0\leq t\leq n-1$, $1\leq g\leq L-1$, $gL^t\leq m\leq\lfloor{L^n/2}\rfloor$, $\xi_m(K_L^n)\geq\xi_{gL^{t}}(K_L^n)$. The proof is completed.
\end{proof}

Lemmas IV.3-IV.8 imply the complex layered-increasing properties and self-similarity of the function $\xi_m(K_L^n)$. It is crucial to $\mathcal{P}$-conditional edge-connectivities of hamming graph $K_L^n$. Analysising the properties of the function $\xi_m(K_L^n)$ might be of interest in their own right.

\section{Unified method for $\mathcal{P}$-conditional edge-connectivities of Hamming graph $K_L^n$}

\begin{figure*}[htbp]
\begin{center}
\scalebox{0.8}{\includegraphics{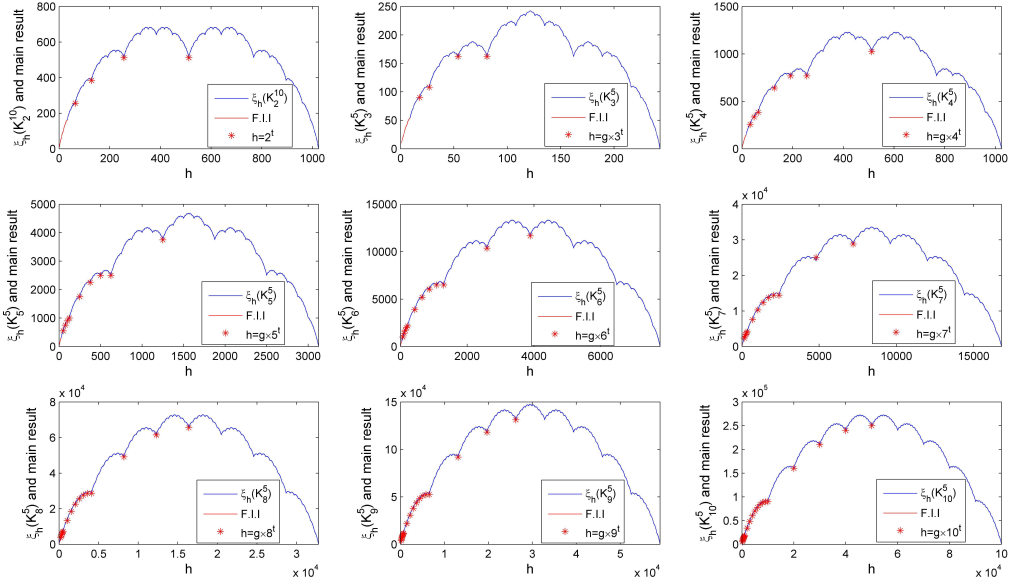}}
\caption{The main results on $\lambda(\mathcal{P}_i^k, K_L^n)$ with $\theta_\mathcal{P}(K_L^n)=gL^t$ and $\theta_\mathcal{P}(K_L^n)\leq L^{\lfloor{n\over 2}\rfloor}$, $\xi_h(K_L^n)$ and $\xi_{gL^t}(K_L^n)$ for $2\leq L\leq 10$.}
\end{center}
\label{fig}
\end{figure*}

\begin{thm}\label{thgL}
Let $n$, $g$, $t$, $L$ be four integers, $n\geq 1$, $L\geq2$, $0\leq t \leq n-1$, $1\leq g\leq L-1$ and $gL^{t}\leq \lfloor{{L^n}/2}\rfloor$. If $\mathcal{P}$-conditional edge-connectivity is bipartite and $\theta_{\mathcal{P}}(K^{n}_{L})=gL^{t}$. We obtain that
$\lambda(\mathcal{P},K^{n}_{L})=\xi_{gL^{t}}(K^{n}_{L})=\delta_{L}(K_{L})ngL^{t}-\delta_{L}(K_{p})gtL^{t}-g(g-1)L^{t}$ $=g[(L-1)(n-t)-(g-1)]L^{t}$.
\end{thm}

\begin{proof}
\textbf{Lower bound:}
	For any conditional $\mathcal{P}$,
	let $F$ be any minimum $\mathcal{P}$-conditional edge-cut of $K^{n}_{L}$, and $K^{n}_{L}-F$ is disconnected. And its deletion from $K^{n}_{L}$ results in $p$ components $C_{1},C_{2}\ldots C_{p}$, $p\geq2$. Each component satisfies the condition $\mathcal{P}$. As $\mathcal{P}$-conditional edge-connectivity of $K^{n}_{L}$ is bipartite, and the minimality of $\lambda(\mathcal{P},K^{L}_{n})$ edge-cut  of $K_L^n$, we have $p=2$. Let $C^{*}$ be the resulting minimum component. Because of $\delta_{\mathcal{P}}(K^{n}_{L})=gL^{t}$, by the \textbf{Lemma} \ref{gld}, we have $|F|\ge |[ V(C^{*}),\overline{V(C^{*})}]|$ $=\xi_{gL^{t}}(K^{n}_{L})=g[(L-1)(n-t)-(g-1)]L^{t}.$

The above equality holds because of $\xi_{gL^{t}}(K^{n}_{L})$ $=\delta_{L-1}(K^{n}_{L})ngL^{t}-(L-1)gtL^{t}-g(g-1)L^{t}=g[(L-1)(n-t)-(g-1)]L^{t}$.

\par \textbf{Upper bound:}
	It is sufficient to show that, there exists a $\mathcal{P}$-conditional edge-cut of $K^{n}_{L}$ with size $\xi_{gL^{t}}(K^{n}_{L})$ $=g[(L-1)(n-t)-(g-1)]L^{t}$, $K^{n}_{L}[L_{gL^{t}}^{n}]=\overbrace{0\ldots 00}^{n-t}X_{t}\ldots$ $X_{2}X_{1}\bigcup\overbrace{0\ldots 01}^{n-t}X_{t}\ldots X_{2}X_{1}\bigcup\dots\overbrace{0\ldots 0g-1}^{n-t}X_{t}\ldots X_{2}X_{1}$.
 Each $\overbrace{0\ldots 0f}^{n-t}X_{t}\ldots X_{2}X_{1}$ for $0\leq f\leq g-1$ is isomorphic to $L$-ary $t$-dimensional sub-layers and $(L-1)t$ regular. The minimum degree of them happens to be $(L-1)t$.
By the definition of $K_L^n$, there exists at least one edge between these $L$-ary $t$-dimensional sub-layers.
So, the induced subgraph $K^{n}_{L}[L_{gL^{t}}^{n}]$ is connected with $gL^{t}$ vertices, and contains a cycle (If $L=2$, $t\geq2$; if $L\geq3$, $t\geq1$). On the other hand, $|\overline{L_{gL^{t}}^{n}}|=L^{n}-gL^{t}=(L-1)L^{n-1}+(L-1)L^{n-2}+\cdots+(L-1)L^{t+1}+(L-g)L^{t}$.
$K^{n}_{L}[\overline{L_{gL^{t}}^{n}}]=\bigcup\limits_{t+1\leq k\leq n-1,1\leq b\leq L-1}
\overbrace{0\ldots 0b}^{n-k}X_{k}\ldots X_{2}X_{1}$ $\bigcup\limits_{g\leq o\leq L-1}\overbrace{0\ldots 0o}^{n-t}X_{t}\ldots X_{2}X_{1}$.

For $t+1\leq k\leq n-1,1\leq b\leq L-1$,
each $\overbrace{0\ldots 0b}^{n-k}X_{k}$ $\ldots X_{2}X_{1}$ is
$L$-ary $k$-dimensional sub-layer of $K_L^n$, which is $(L-1)k$ regular and isomorphic to $K_{L}^{k}$. There is at least one edge between different sub-layers.
For each $g\leq o\leq L-1$, $\overbrace{0\ldots 0o}^{n-t}X_{t}\ldots X_{2}X_{1}$
is $L$-ary $t$-dimensional sub-layer of $K_L^n$,  which is $(L-1)t$ regular and isomorphic to $K_{L}^{t}$. Also there is at least one edge between different sub-layers.
And $\overbrace{0\ldots 0l}^{n-t-1}X_{t+1}\ldots X_{2}X_{1}$ are connected with $\overbrace{0\ldots 0L-1}^{n-t}X_{t}\ldots X_{2}X_{1}$ by the definition of $K_{L}^{n}$. So $K^{n}_{L}[\overline{L_{gL^{t}}^{n}]}$ is connected with $L^{n}-gL^{t}$ vertices ($\geq gL^{t}$), and contains at least a cycle.
As $\mathcal{P}$-conditional edge-connectivity is bipartite and $\theta_{\mathcal{P}}(K^{n}_{L})=gL^{t}$. Thus, $[L_{gL^{t}}^{n}, \overline{L_{gL^{t}}^{n}}]$ is a $\mathcal{P}$-conditional edge-cut of $K^{n}_{L}$.
So $|[L_{gL^{t}}^{n}, \overline{L_{gL^{t}}^{n}}]|=\xi_{gL^{t}}(K^{n}_{L})=g[(L-1)(n-t)-(g-1)]L^{t}$.
\end{proof}
\begin{table*}[htbp]\scriptsize
\caption{$\mathcal{P}_i^k$-conditional edge-connectivity $\lambda(\mathcal{P}_i^k, K_L^n)$.}
\begin{center}
\begin{tabular*}{\hsize}{@{}@{\extracolsep{\fill}}lllllllllll}
\hline
\hline
$\lambda(\mathcal{P}_i^k, K_L^n)$&$t$&$K_2^n=Q_n$&$K_3^n$&$K_4^n$&$\cdots$&$K_L^n$\\
$\lambda(\mathcal{P}_1^{L^t}, K_L^n)$&$0\leq t\leq n-1$&$(n-t)2^t$&$2(n-t)3^t$&$3(n-t)4^t$&$\cdots$&$(L-1)(n-t)L^t$\\
$\lambda(\mathcal{P}_2^{t}, K_L^n)$&$0\leq t\leq n-1$&$(n-t)2^t$&$2(n-t)3^t$&$3(n-t)4^t$&$\cdots$&$(L-1)(n-t)L^t$\\
$\lambda(\mathcal{P}_3^{c}, K_L^n)$&&$(n-2)2^2$&$2(n-1)3^1$&$3[(n-0)-2]4^0$&$\cdots$&$3[(L-1)(n-0)-2]L^0$\\
$\lambda(\mathcal{P}_4^{(L-1)t}, K_L^n)$&$0\leq t\leq n-1$&$(n-t)2^{t}$&$2(n-t)3^{t}$&$3(n-t)4^{t}$&$\cdots$&$(L-1)(n-t)L^t$\\
$\lambda(\mathcal{P}_5^{(L-1)t}, K_L^n)$&$0\leq t\leq n-1$&$(n-t)2^{t}$&$2(n-t)3^{t}$&$3(n-t)4^{t}$&$\cdots$&$(L-1)(n-t)L^t$\\
$\lambda(\mathcal{P}_6^{L^t}, K_L^n)$&$0\leq t\leq n-1$&$(n-t)2^{t}$&$2(n-t)3^{t}$&$3(n-t)4^{t}$&$\cdots$&$(L-1)(n-t)L^t$\\
\hline
\hline
\end{tabular*}
\label{tab3}
\end{center}
\end{table*}

By \textbf{Theorem} \ref{thgL}, the exact values of
$L^{t}$-extra edge-connectivity, $t$-embedded edge-connectivity, cyclic edge-connectivity, $(L-1)t$-super edge-connectivity, $(L-1)t$-average edge-connectivity and $L^{t}$-th isoperimetric edge-connectivity of $K_L^n$ share the same values in form of $(L-1)(n-t)L^{t}$ because they are all bipartite.

Based on the \textbf{Theorem} \ref{thgL} and the following \textbf{Theorem} \ref{th2}, an algorithm can be designed to calculate bipartite $\mathcal{P}$-conditional edge-connectivity of $K_L^n$ with $\theta_{\mathcal{P}(K_L^n)}(K_L^n)=gL^t$ or $\theta_{\mathcal{P}(K_L^n)}(K_L^n)\leq L^{\lfloor{n\over2}\rfloor}$ in the \textbf{Algorithm.1}. In this algorithm,  $O(K_L)=L^{\lfloor{L\over2}\rfloor}$. The time complexity of the algorithm is $O(log_L(N))$, where $N=L^n$.

\begin{cor}\label{cor1}
For any integers $0\leq t\leq n-1$, $\lambda(\mathcal{P}_1^{L^t},K_L^n)$, $\lambda(\mathcal{P}_2^{t},K_L^n)$,
$\lambda(\mathcal{P}_4^{(L-1)t},K_L^n)$, $\lambda(\mathcal{P}_5^{(L-1)t},K_L^n)$ and $\lambda(\mathcal{P}_6^{L^t},K_L^n)$ share the same value $\xi_{L^{t}}(K^{n}_{L})=(L-1)(n-t)L_{t}$.
And $\lambda(\mathcal{P}_3^c,K_L^n)=g[(L-1)(n-t)-(g-1)]L^{t}$.
\end{cor}
\begin{proof}
As the hamming graph $K_L^n$ is $(L-1)n$ regular, with $L^n$ vertices,
for any integers $0\leq t\leq n-1$, $\lambda(\mathcal{P}_1^{L^t},K_L^n)$, $\lambda(\mathcal{P}_2^{t},K_L^n)$, $\lambda(\mathcal{P}_3^c,K_L^n)$,
$\lambda(\mathcal{P}_4^{(L-1)t},K_L^n)$, $\lambda(\mathcal{P}_5^{(L-1)t},K_L^n)$ and $\lambda(\mathcal{P}_6^{L^t},K_L^n)$ are well-defined, and then the set \{$\mathcal{P}_1^{L^t}$, $\mathcal{P}_2^{t}$, $\mathcal{P}_3^c$,
$\mathcal{P}_4^{(L-1)t}$, $\mathcal{P}_5^{(L-1)t}$, $\mathcal{P}_6^{L^t} \}\subseteq \mathcal{B}$.
By \textbf{Theorem} \ref{thgL}, the exact values of
$\lambda(\mathcal{P}_1^{L^t},K_L^n)$ $\lambda(\mathcal{P}_2^{t},K_L^n)$,
$\lambda(\mathcal{P}_4^{(L-1)t},K_L^n)$, $\lambda(\mathcal{P}_5^{(L-1)t},K_L^n)$ and $\lambda(\mathcal{P}_6^{L^t},K_L^n)$ share the same value $\xi_{L^{t}}(K^{n}_{L})=(L-1)(n-t)L^{t}$, which equal to the minimum number of links-faulty resulting in an $L$-ary $t$-dimensional sub-layer with $L^{t}$ vertices  from $K_{L}^{n}$. Note that the value of $g$ is $1$, but for the case of cyclic edge-connectivity, we have $g\leq 3$. Actually, for $K_2^n$, the shortest cycle is $C_4$ with length $L^2=2^2$ for $g=1$, $t=2$, and for $K_3^n$, the shortest cycle is $C_3$ with length $L^1=3^1$ for $g=1$, $t=1$, while for $K_L^n$, $L\geq4$, the shortest cycle is $C_3$ with length $3L^0=3\times L^0$ for $g=3$, $t=0$. So $\lambda(\mathcal{P}_3^c,K_L^n)=g[(L-1)(n-t)-(g-1)]L^{t}$.
The proof is finished.
\end{proof}

\begin{algorithm}[!h]{\normalsize
\DontPrintSemicolon
\SetAlgoLined
\KwResult{Calculating the $h$-extra edge-connectivity of $K_L^n$}
\SetKwInOut{Input}{Input}\SetKwInOut{Output}{Output}
\Input{Given positive integers $n$, $L$, $i$, $k$ and $K_L^n$, $\theta_\mathcal{P}(K_L^n)$,$O(K_L)$}
\Output{The $\mathcal{P}$ conditional edge connectivity of $K_L^n$ is $S$}
\BlankLine
$I_0\leftarrow 0$;
$v\leftarrow 1$;
\While{$v\leq L$}{
$I_v\leftarrow {{v(v-1)}\over2}$;
$v\leftarrow v+1$;
$\delta_v\leftarrow v-1$;
}
\If{$\theta_\mathcal{P}(K_L^n)\leq O(K_L)$}{
        {$m\leftarrow\theta_\mathcal{P}(K_L^n)$;\;}
}
\If{$\theta_\mathcal{P}(K_L^n)==gL^t$}{{
$m\leftarrow gL^t$;}
}
$t\leftarrow m$;\;
$s'\leftarrow\lfloor{\log_Lt}\rfloor+1$;\;
\While{$s'==s'-1$}
{
  \eIf{$t\leq L$}
  {$c(s')\leftarrow t$;break}
  {$c(s')\leftarrow\mod(t,L)$;
   $t\leftarrow\lfloor{t\over L}\rfloor$;
   $s'\leftarrow s'-1$;
  }
}
{$k\leftarrow1$;
$z\leftarrow1$;\;}
\While{$k\leq\lfloor{\log_Lm}\rfloor+1$}
{\eIf{$c(k)\neq0$}
  {
  $a(z)\leftarrow c(k)$;
  $b(z)\leftarrow\lfloor{\log_Lm}\rfloor+1-k$;
  $z\leftarrow z+1$;
  }
  { $k\leftarrow k+1$;}
}
{$q\leftarrow 1$;
$S\leftarrow \delta_Lnm$;\;}
\While{$q\leq z$}
{
  \eIf{$q\leftarrow1$ or $q\leftarrow z$}
  {
  $S\leftarrow S-\delta_La(q)b(q)L^{b(q)}-I_{a(q)}L^{b(q)}$;
  $q\leftarrow q+1$;
  }
  {
  $S\leftarrow S-\delta_La(q)b(q)L^{b(q)}-I_{a(q)}L^{b(q)}$;
  $w\leftarrow q+1$;\;
  \While{$w\leq z$}
    {
    $S\leftarrow S-2\delta_{a(w)}a(w)L^{b(w)}$;
    $w\leftarrow w+1$;
    }
  $q\leftarrow q+1$;
  }
}
\caption{$\mathcal{P}$-conditional edge-connectivity of $K^{n}_{L}$}}
\end{algorithm}

\begin{cor}\label{cor1}
For any integers $0\leq t\leq n-1$, \\
1). \cite{em21}(2021)$\lambda(\mathcal{P}_2^{t},K_3^n)=\eta_{t}(K_3^n)=\xi_{3^{t}}(K^{n}_{3})=2(n-t)3^{t}$;\\
2). \cite{lx16}(2012)$\lambda(\mathcal{P}_2^{t},K_2^n)=\eta_{t}(K_2^n)=\xi_{2^{t}}(K^{n}_{2})=(n-t)2^{t}$;\\
3). \cite{lx13}(2013)$\lambda(\mathcal{P}_2^{t},\mathcal{B}_n)=\eta_{t}(\mathcal{P}_2)=\xi_{2^{t}}(\mathcal{B}_n)=(n-t)2^{t}$.
\end{cor}

For $2\leq L\leq 10$, the main results of $\mathcal{P}_i^k$-conditional edge-connectivity of $K_L^n$ with $\theta_\mathcal{P}(K_L^n)=gL^t$ and $\theta_\mathcal{P}(K_L^n)\leq L^{\lfloor{n\over 2}\rfloor}$, $\xi_h(K_L^n)$ and $\xi_{gL^t}(K_L^n)$ are shown in the Fig.2, where F.I.I represents the exact values of function $\xi_h(K_L^n)$ in the first increasing interval $1\leq h \leq L^{\lfloor{n\over2}\rfloor}$. Once the $\mathcal{P}$-conditional edge-connectivity of $K_L^n$ is bipartite and $\theta_{\mathcal{P}}(K_L^n)=\min\{|X|\,|\,K_L^n[X]$ satisfying the property $\mathcal{P}\}$ is less than $L^{\lfloor{n\over2}\rfloor}$, the exact value of $\mathcal{P}$-conditional edge-connectivity of $K_L^n$ is $\xi_{\theta_{\mathcal{P}(K_L^n)}}(K_L^n)$.
We also investigate the cases of bipartite $\mathcal{P}$-conditional edge-connectivity of $K_L^n$ with $\theta_{\mathcal{P}(K_L^n)}(K_L^n)=gL^t$.
The exact values of $\mathcal{P}_i^k$-conditional edge-connectivity of $K_L^n$, $\lambda(\mathcal{P}_i^k, K_L^n)$, are shown in Table \ref{tab3}.

\begin{thm}\label{th2}
Let $n$ and $L$ be two positive integers. If $\mathcal{P}$-conditional edge-connectivity of $K^{n}_{L}$ is bipartite with $h=\theta_{\mathcal{P}}(K^{n}_{L})\leq L^{\lfloor\frac{n}{2}\rfloor}$. Then, one can obtain that $\lambda(\mathcal{P},K^{n}_{L})=\xi_{\theta_{\mathcal{P}}(K^{n}_{L})}(K^{n}_{L})=\xi_{h}(K^{n}_{L})=(L-1)nh-ex_{h}(K^{n}_{L})$ where $h=\sum\nolimits_{i=0}^{s}a_{i}L^{b_{i}}$, $ex_{h}(K^{n}_{L})=\sum\nolimits_{i=0}^{s}[(L-1)a_{i}b_{i}L^{b_{i}}+(a_{i}-1)a_{i}L^{b_{i}}]+2\sum\nolimits_{i=0}^{s-1}\sum\nolimits_{j=i+1}^{s}a_{i}a_{j}L^{b_j}$.
\end{thm}
\begin{proof}
As for each $1\leq m\leq\lfloor L^{n}/2\rfloor$, one can find a subset $X_{m}^{*}=L_{m}^{n}\subset V(K^{n}_{L})$ satisfying that $\xi_{m}(K^{n}_{L})=|[L_{m}^{n},\overline{L_{m}^{n}}]|, |L_{m}^{n}|=m$ and both $K^{n}_{L}[L_{m}^n]$ and $K^{n}_{L}[\overline{L_{m}^n}]$ are connected with $2|E(K^{n}_{L}[L_{m}^n])|=ex_{m}(K^{n}_{L})$. Because $\mathcal{P}$-conditional edge-connectivity is bipartite, for any minimum $\mathcal{P}$-conditional edge-cut of $K^{n}_{L}$ with $m^{*}$ vertices for smaller component for some $m^{*}\leq \lfloor L^{n}/2\rfloor$. By \textbf{Lemma} \ref{Znz}, \textbf{Lemma} \ref{gld} and $\theta_{\mathcal{P}}(K^{n}_{L})=h\leq L^{\lfloor\frac{n}{2}\rfloor}$, we can deduce that $\lambda(\mathcal{P},K^{n}_{L})\geq \xi_{\theta_{\mathcal{P}}(K^{n}_{L})}(K^{n}_{L})=\xi_{h}(K^{n}_{L})$. Because $[L_{\theta_{\mathcal{P}}(K^{n}_{L})}^{n}, \overline{L_{\theta_{\mathcal{P}}(K^{n}_{L})}^{n}}]$ is a $\mathcal{P}$-conditional edge-cut, $\lambda(\mathcal{P},K^{n}_{L})\leq|[L_{\theta_{\mathcal{P}}(K^{n}_{L})}^{n}, \overline{L_{\theta_{\mathcal{P}}(K^{n}_{L})}^{n}}]|=|[L_{h}^{n}, \overline{L_{h}^{n}}]|=\xi_{h}(K_{L}^{n})=(L-1)nh-ex_{m}(K_{L}^{n})$.
So $\lambda(\mathcal{P},K^{n}_{L})=\xi_{\theta_{\mathcal{P}}(K^{n}_{L})}(K^{n}_{L})=\xi_{h}(K^{n}_{L})=(L-1)nh-ex_{h}(K^{n}_{L})$.
The result holds.
\end{proof}

\begin{cor}\label{cor2}
Let $n$, $h$ and $L$ be three positive integers, $h\leq L^{\lfloor\frac{n}{2}\rfloor}$. The $h$-extra edge-connectivity of $K^{n}_{L}$ is $\lambda_{h}(K^{n}_{L})=\lambda_{h}(K^{n}_{L})=(L-1)nh-ex_{h}(K^{n}_{L})$
where $h=\sum\nolimits_{i=0}^{s}a_{i}L^{b_{i}}$, $ex_{h}(K^{n}_{L})=\sum\nolimits_{i=0}^{s}(L-1)a_{i}b_{i}L^{b_{i}}+(a_{i}-1)a_{i}L^{b_{i}}+2\sum\nolimits_{i=0}^{s-1}\sum\nolimits_{j=i+1}^{s}a_{i}a_{j}L^{b_j}$.
\end{cor}
Our results improve the case for $L=2,3$, $h\leq 2^{\lfloor\frac{n}{2}\rfloor}$.
\begin{cor}\label{cor1}
For any integer $0\leq t\leq n-1$, $L=2,3$, $h\leq L^{\lfloor\frac{n}{2}\rfloor}$.
\par(1)~\cite{Zhang}(2018)$\lambda(\mathcal{P}_1^{h},K_3^n)=\lambda_{h}(K_3^n)=\xi_{h}(K^{n}_{3})$;
\par(2)~\cite{Li(2013)}(2013)$\lambda(\mathcal{P}_1^{t},K_2^n)=\lambda_{h}(K_2^n)=\xi_{h}(K^{n}_{2})$;
\par(3)~\cite{Zhang(2014)}(2014)$\lambda(\mathcal{P}_1^{h},\mathcal{B}_n)=\lambda_{h}(\mathcal{B}_n)=\xi_{h}(\mathcal{B}_n)$;
\par(4)~\cite{hl}(2013)$\lambda(\mathcal{P}_1^{4},\mathcal{B}_n)=\lambda_{4}(\mathcal{B}_n)=\xi_{4}(\mathcal{B}_n)=4n-8$.
\end{cor}

\section{Conclusion}
Reliability evaluation and fault tolerance of an interconnection network of some parallel and distributed systems are discussed separately under various link-faulty hypotheses in terms of different $\mathcal{P}$-conditional edge-connectivity, where $\mathcal{P}$ is some graph-theoretic property of a connected graph $G$.
This paper deals with the $\mathcal{P}$-conditional edge-connectivities of hamming graph $K_{L}^{n}$ with satisfying the property that each minimum $\mathcal{P}$-conditional edge-cut separates the $K_{L}^{n}$ just into two components. And these $\mathcal{P}$-conditional edge-connectivity is called bipartite. We show that for hamming graph $K_{L}^{n}$, the $L^{t}$-extra edge-connectivity, $t$-embedded edge-connectivity, $(L-1)t$-super edge-connectivity, $(L-1)t$-average edge-connectivity and $L^{t}$-th isoperimetric edge-connectivity share the same values in form of $(L-1)(n-t)L^{t}$.
Besides, we also obtain the exact values of $h$-extra edge-connectivity and $h$-th isoperimetric edge-connectivity of hamming graph  $K_{L}^{n}$ for each $h\leq L^{\lfloor {\frac{n}{2}} \rfloor}$. Among these six kinds of connectivity, the solution of $h$-extra edge-connectivity is crucial for solving the other five or even more general $\mathcal{P}$-conditional edge-connectivity. Because if any abstract $\mathcal{P}$-conditional edge-connectivity with bipartite property, the smallest cut must be divided into two components, which is composed of a component and its complement. The exact value of the $\mathcal{P}$-conditional edge-connectivity can be determined according to the size of the $\lambda(\mathcal{P})$-atom. The value of $h$-extra edge-connectivity is just a description of the number vertices in the component.
Our results improve several previous results on this topic and can be applied to bijective connection networks, which contain hypercubes, twisted cubes, crossed
cubes, M\"obius cubes, locally twisted cubes and so on.


\begin{thebibliography}{23}
\bibitem{tc} F.B. Chedid, On the generalized twisted cube, Information Processing Letters, 55 (1) (1995) 49-52.
\bibitem{gc} G. Chartrand, S. F. Kapoor, L. Lesniak and D. R. Lick, Generalized connectivity in graphs, Bull. Bombay Math. Colloq. 2 (1984) 1-6.
\bibitem{cdh2015} N. Chang, W. Deng, and S. Hsieh, Conditional diagnosability of $(n, k)$-star networks under the comparison diagnosis model, IEEE Transactions on Reliability, 64 (1) (2015) 132-143.
\bibitem{mc} P. Cull, S. Larson, The M$\ddot o$bius cubes, IEEE Transactions on Computers 44 (1995) 647-659.
\bibitem{hld} K. Efe, A variation on the hypercube with lower diameter, IEEE Transactions on Computers 40 (11) (1991) 1312-1316.
\bibitem{Foil(1996)} J. F\`{a}brega and M. Foil, On the extra connectivity of graphs,
Discrete Mathematics, 155 (1-3) (1996) 49-57.
\bibitem{Harary} F. Harary, Conditional Connectivity, Networks, 13 (3) (1983) 347-357.

\bibitem{harper1964optimal} L.H. Harper, Optimal assignments of numbers to vertices, J. Soc. Indust. Appl. Math. 12 (1) (1964) 131-135.

\bibitem{tc2} P.A.J. Hibers, M.R.J. Koopman and J.V.D. Snepscheut, The twisted cube, Proceedings of the Conference on Parallel Architectures and Languages Europe, Lecture Notes in Computer Science, Springer, 258 (1987) 152-159.
\bibitem{hl} W. Hong and S. Hsieh, Extra edge connectivity of hypercube-like networks, International Journal of Parallel, Emergent and Distributed Systems, 28 (2) (2013) 123-133.
\bibitem{hh2012} W. Hong and S. Hsieh, Strong diagnosability and conditional diagnosability of augmented cubes under the comparison diagnosis model,
IEEE Transactions on Reliability, 61 (1) (2012) 140-148.
\bibitem{hl2000} Y.O. Hamidoune, A.S. Llad\'{o}, O. Serra and R. Tindell, On isoperimetric connectivity in vertex-transitive graphs, Siam Journal on Discrete Mathematics, 13 (2000) 139-144.

\bibitem{lx16} X.-J. Li, Q.-Q. Dong, Z. Yan, J.-M. Xu, Embedded connectivity of recursive networks, Theoretical Computer Science, 653 (2016) 79-86.
\bibitem{lx13} X.-J. Li, J.-M. Xu, Edge-fault tolerance of hypercube-like networks, Information Processing Letters, 113 (19-21) (2013) 760-763.
\bibitem{lxc} L. Lin, L. Xu, R. Chen, S. Hsieh, and D. Wang, Relating extra connectivity and extra conditional diagnosability in regular networks, IEEE Transactions on Dependable and Secure Computing, 16 (6) (2017) 1086-1097.
\bibitem{lxz2016} L. Lin, L. Xu, and S. Zhou, Relating the extra connectivity and the
conditional diagnosability of regular graphs under the comparison model,
Theoretical Computer Science, 618 (7) (2016) 21-29.
\bibitem{lxz} L. Lin, L. Xu, S. Zhou, and S. Hsieh, The extra restricted connectivity and conditional diagnosability of split-star networks, IEEE Transactions on Parallel and Distributed Systems, 27 (2) (2016) 533-545.
\bibitem{lxz2015} L. Lin, L. Xu, S. Zhou, and D. Wang, The reliability of subgraphs in
the arrangement graph, IEEE Transactions on Reliability, 64 (2) (2015) 807-818.
\bibitem{Li(2013)} H. Li and W. Yang, Bounding the size of the subgraph induced by $m$ vertices and extra edge-connectivity of hypercubes, Discrete Applied Mathematics, 161 (16-17) (2013) 2753-2757.
\bibitem{lzxw2015} L. Lin, S. Zhou, L. Xu, and D. Wang, The extra connectivity and conditional diagnosability of alternating group networks, IEEE Transactions on Parallel and Distributed Systems, 26 (8) (2015) 2352-2362.
\bibitem{ms2017} L.P. Montejano and I. Sau, On the complexity of computing the $k$-restricted edge-connectivity of a graph, Theoretical Computer Science, 662 (1) (2017) 31-39.
\bibitem{cc} Plummer MD, On the cyclic connectivity of planar graphs, Lect Notes Math 303 (1972) 235-242.
\bibitem{mm1971} W. Mader, Minimale $n$-fach Kantenzusammenh$\ddot a$ngenden Graphen, Mathematische Annalen, 191 (1971) 21-28.
\bibitem{m} N.K. Singhvi and K. Ghose, The Mcube: a symmetrical cube based network with twisted links, Proceedings of the 9th IEEE International Parallel Processing Symposium, (1995) 11-16.
\bibitem{t1996} R. Tindell, Connectivity of Cayley graphs, D.Z. Du, D.F. Hsu (Eds.), Combinatorial Network Theory, (1996) 41-64.
\bibitem{xlzh2016} L. Xu, L. Lin, S. Zhou, and S. Hsieh, The extra connectivity, extra conditional diagnosability, and $t/m$-diagnosability of arrangement graphs,
IEEE Transactions on Reliability, 65 (3) (2016) 1248-1262.
\bibitem{ltc} X. Yang, D. Evans, G.M. Megson, The locally twisted cubes, Internationa Journal of  Computer Mathematics 82 (4) (2005) 401-413.
\bibitem{em} Y. Yang and S. Wang, Conditional connectivity of star graph networks under embedding restriction, Information Sciences, 199 (2012) 187-192.
\bibitem{em21}Y. Yang, Embedded connectivity of ternary $n$-cubes, Theoretical Computer Science, 871 (2021) 121-125.
\bibitem{Zhang(2014)} M. Zhang, J. Meng, W. Yang, and Y. Tian, Reliability analysis of bijective connection networks in terms of the extra edge-connectivity, Information Sciences, 279 (2014) 374-382.
\bibitem{Zhang} M. Zhang, Edge isopermetric problem on graphs and the related applications[D], Xia-men: University of Xiamen, (2018) 68-77.
\bibitem{Zhang(2018)} M. Zhang, L. Zhang, X. Feng and H. Lai, An $O(log_2(N))$ algorithm for reliability evaluation of $h$-extra edge-connectivity of folded hypercubes, IEEE Transactions on Reliability, 67 (1) (2018) 297-307.
\bibitem{x} X. Zhu, The $Z$-cubes: a hypercube variant with small diameter, Journal of graph theory, 85 (3) (2017) 651-660.
\bibitem{z} Z. Zhang, Extra edge connectivity and isoperimetric edge connectivity, Discrete Mathematics, (2008) 4560-4569.

\end{thebibliography}
\end{document}